\documentclass[final,leqno,onefignum,onetbum]{siamart1116}
\usepackage{graphicx}
\usepackage{amssymb}    
\usepackage{epsf}
\usepackage{subfigure}
\usepackage{color}
\usepackage{latexsym}

\usepackage{mathtools}

\usepackage{algorithm}
\usepackage{algpseudocode}
\usepackage{lscape}
\usepackage{fullpage}
\usepackage[titletoc]{appendix}
\usepackage{fancyhdr}
\usepackage{amsmath,amssymb,xspace}
\usepackage{booktabs}
\usepackage{tabularx}

\newfont{\bbb}{msbm10 scaled\magstep1}

\DeclareFontEncoding{FMS}{}{}
\DeclareFontSubstitution{FMS}{futm}{m}{n}
\DeclareFontEncoding{FMX}{}{}
\DeclareFontSubstitution{FMX}{futm}{m}{n}
\DeclareSymbolFont{fouriersymbols}{FMS}{futm}{m}{n}
\DeclareSymbolFont{fourierlargesymbols}{FMX}{futm}{m}{n}
\DeclareMathDelimiter{\VERT}{\mathord}{fouriersymbols}{152}{fourierlargesymbols}{147}

\newcommand{\ri}{\mathrm{i}}
\newcommand{\rd}{\mathrm{d}}
\newcommand{\R}{\mathbb{R}}

\newcommand{\N}{\mathbb{N}}

\newcommand{\tr}{\mathrm{Tr}}

\newcommand{\ie}{\textit{i.e.}\;}
\newcommand{\eg}{\textit{e.g.}\;}

\newcommand{\cunit}{~\text{m/s}}
\newcommand{\xunit}{~\text{m}}

\newcommand{\dds}[1]{\frac{\mathrm{d}#1}{\mathrm{d}s}}

\newcommand{\p}{\partial}
\newcommand{\f}[2]{\frac{#1}{#2}}

\newcommand{\Dt}{\Delta}
\newcommand{\dt}{\delta}

\newcommand{\Og}{\Omega}

\newcommand{\veps}{\varepsilon}

\newcommand{\bl}{\begin{law}}
	\newcommand{\el}{\end{law}}
\newcommand{\bthm}{\begin{thm}}
	\newcommand{\ethm}{\end{thm}}

\newcommand{\mcl}{\mathcal}

\newcommand{\dpm}{\displaystyle}

\newcommand{\abs}[1]{\left\lvert#1\right\rvert}

\newcommand{\bpm}{\left(\begin{array}}
	\newcommand{\epm}{\end{array}\right)}

\newcommand{\tu}{\tilde{u}}
\newcommand{\tua}{\tu^\dagger}

\newcommand{\K}{K}
\newcommand{\tK}{\tilde{K}}

\newcommand{\dG}{{\dot{G}}}

\newcommand{\E}{\mathbb{E}}
\newcommand{\Pb}{\mathbb{P}}

\newcommand{\mis}{\mathrm{J}}
\newcommand{\X}{\mathrm{X}}
\newcommand{\tX}{\tilde{\mathrm{X}}}
\newcommand{\Y}{\mathrm{Y}}
\newcommand{\BatG}{\mathfrak{B}}
\newcommand{\BatR}{\mathfrak{R}}
\newcommand{\BatS}{\mathfrak{S}}

\newcommand{\obs}{\mathrm{obs}}

\newcommand{\rec}{\mathrm{R}}
\newcommand{\src}{\mathrm{S}}

\usepackage{mathrsfs}

\title{Seismic Tomography with Random Batch Gradient Reconstruction}

\begin{document}
	
	\author{
		Yixiao Hu\footnote{Department of Mathematical Sciences, Tsinghua University, Beijing, 100084, China (yx-hu16@mails.tsinghua.edu.cn).}
		\and
		Lihui Chai\footnote{School of Mathematics, Sun Yat-sen University, Guangzhou, 510275, China (chailihui@mail.sysu.edu.cn).} 
		\and 
		Zhongyi Huang\footnote{Department of Mathematical Sciences, Tsinghua University, Beijing, 100084, China (zhongyih@tsinghua.edu.cn).}  
		\and Xu Yang\footnote{Department of Mathematics, University of California, Santa Barbara, CA 93106, USA (xuyang@math.ucsb.edu).}
	}
	
	\begingroup
	
	
	\makeatother
	
	\endgroup
	
	
	\maketitle
	\begin{abstract}
		Seismic tomography solves high-dimensional optimization problems to image subsurface structures of Earth. In this paper, we propose to use random batch methods to construct the gradient used for iterations in seismic tomography. Specifically, we use the frozen Gaussian approximation to compute seismic wave propagation, and then construct stochastic gradients by random batch methods. The method inherits the spirit of stochastic gradient descent methods for solving high-dimensional optimization problems. The proposed idea is general in the sense that it does not rely on the usage of the frozen Gaussian approximation, and one can replace it with any other efficient wave propagation solvers, {\it e.g.}, Gaussian beam methods and spectral element methods. We prove the convergence of the random batch method in the mean-square sense, and show the numerical performance of the proposed method by two-dimensional and three-dimensional examples of wave-equation-based travel-time inversion and full-waveform inversion, respectively. As a byproduct, we also prove the convergence of the accelerated full-waveform inversion using dynamic mini-batches and spectral element methods.

	\end{abstract}
	
	
	

	\newtheorem{thm}{Theorem}[section]
	\newtheorem{lem}{Lemma}[section]
	\newtheorem{coro}{Corollary}[section]
	\newtheorem{prop}{Proposition}[section]
	\newtheorem{define}{Definition}[section]
	\newtheorem{remark}{Remark}[section]
	\newtheorem{assumption}{Assumption}

	\section{Introduction}
	
	Seismic tomography can provide crucial information by computing images for the subsurface structures of Earth at different scales, and the understanding of tectonics, volcanism, and geodynamics \cite{Aki1976, Romanowicz1991, Rawlinson2010, Zhao2012a}. Wave-equation-based seismic tomography solves the nonlinear high-dimensional optimization problem iteratively for velocity models by computing seismograms and sensitivity kernels in complex models \cite{Tromp2005, Liu2008, Liu2012, Tong2014a}. Successful applications include imaging the velocity models of the southern California crust \cite{Tape2009, Tape2010}, the European upper mantle \cite{Zhu2012}, {the North Atlantic region \cite{Rickers2013}, and the Japan islands \cite{Simute2016}}. The performance of seismic tomography is restricted by how accurate and efficient one can compute synthetic seismograms and sensitivity kernels \cite{Tromp2005}, which are used to construct descent directions of velocity models for iterations. 
	The computation of the synthetic seismograms could be extremely expensive for large-scale and high-frequency 3-D simulations, and
	the velocity models live in a high-dimensional space, making it challenging to compute its descent directions and find the global minima, known as the curse of dimensionality. There have been recent research works aiming to overcome the challenge, \eg, using a randomized optimizer to search the global minima \cite{LiLinYa:14}, and using the Wasserstein metric to improve the convexity \cite{qiu2017full,YangEng2018}.

	Stochastic gradient descent (SGD), frequently used in training deep neural networks, has been proved to be efficient in solving high-dimensional optimization problems. Note that a common gradient descent needs to accurately compute gradients in high dimensions at each iteration, which can be computationally prohibitive and stuck in bad local minima, while SGD can perform better in overcoming these issues illustrated, for example, by the Adam method \cite{kingma2014adam}. Motivated by the success of SGD, we propose to solve the high-dimensional optimization problem in seismic tomography by constructing descent directions of velocity models using the random batch method (RBM), recently proposed for computing dynamics of interacting particles \cite{jin2020random}. 
	
	To use RBM in seismic tomography, possible choices for computing synthetic seismograms are numerical methods of particle type, {\it e.g.}, generalized ray theory \cite{He:68,ViHe:88}, Kirchhoff migration \cite{Gr:86,KeBe:88}, Gaussian beam
	migration \cite{Hi:90,Hi:01,SEG-2003-11141117,Gr:05,GrBe:09,PoSePoVe:10}, 
	Gaussian beam method \cite{Qian2010,Bao2013}, 
	and frozen Gaussian approximation (FGA) \cite{LuYa:11,chai2017frozen, chai2018tomo,hate2018fga,chai2021fgaconv}. One may also use direct numerical methods ({\it e.g.}, the spectral element method) to compute the synthetic sesmograms if the random batch is applied to source-receiver pairs \cite{van2020accelerated}. Here for the sake of convenience, we use FGA to compute wave equations.
	FGA was originally used in quantum chemistry for Schr\"odinger equation \cite{He:81,HeKl:84}, with systematic justifications in \cite{Ka:94,Ka:06,SwRo:09}. Then the formula was generalized to linear hyperbolic systems \cite{LuYa:11,LuYa:CPAM}, with applications in seismic tomography \cite{YaLuFo:13,LiLinYa:14,chai2017frozen,chai2018tomo,hate2018fga}.
	FGA does not need to solve ray paths by shooting to reach the receivers, and can provide accurate solutions in the presence of caustics and multipathing, with no requirement on tuning beam width parameters to achieve a good resolution  \cite{CePoPs:82,Hi:90,FoTa:09,Qian2010,LuYa:11}.  
	
	In this paper, we focus on seismic tomography based on acoustic wave propagation (P wave). We shall study both the wave-equation-based travel-time inversion (TTI) and full-waveform inversion (FWI). Specifically, we compute the wave equations by FGA and construct the sensitivity kernel by RBM, yielding the stochastic decent directions for the velocity model. Then the convergent iterations in TTI and FWI will produce the velocity model expected in seismic tomography. We analyze the convergence of the proposed method in the mean-square sense, and show the accuracy by a two-dimensional (2D) Gaussian perturbation model, a 2D gradually changing background model, a 2D three-layered model, and a three-dimensional (3D) Gaussian perturbation model. 
	
	The rest of the paper is organized as follows. In Section~\ref{sec:model}, we introduce the model setup and the formulation of seismic tomography. In Section~\ref{sec:randombatch}, we systematically describe the construction of stochastic gradients by frozen Gaussian approximation and random batch method, and then prove the convergence to the deterministic gradient descent in the mean-square sense. 
	In addition, we generalize the idea and provide convergence analysis for the accelerated full-waveform inversion method \cite{van2020accelerated} which performs the random batch on source-receiver pairs. We present the numerical performance by several examples in Section~\ref{sec:example}, and make conclusive remarks in Section~\ref{sec:conclusion}.

	\section{Seismic tomography}\label{sec:model}
	In this section, we introduce the formulation of seismic tomography where the propagation of seismic waves is modeled by acoustic wave equations (P-wave),
	\begin{align}\label{eqn:wave3d}
	\begin{split}
		&\rho(x)\,\p_t^2u-\Dt_{x} u=s(t,x), ~\quad x\in\R^3, ~t>0,
		\\
		&u(t=0,x) = 0, \quad \p_t u(t=0,x) = 0.
	\end{split}
	\end{align}
	Here $\rho(x)$ is the reference media density at location $x\in\R^3$, from which one can get the P-wave velocity $c=1/\sqrt\rho$, $\Dt_x$ is the Laplace operator in $x$, and $s(t,x)$ is the source term. When the earthquake is modeled by a point source, one can choose $s(t,x)=f(t)\dt_d(x-x_s)$ with $f(t)$ as the source time function at ${x}_s$ with compact support on $[0,\infty)$, and $\dt_d$ as the Dirac delta function. Remark that the formulation here can be easily generalized to elastic wave propagation as in, {\it e.g.}, \cite{Tromp2005, hate2018fga}, and we focus on seismic tomography using the propagation of P-wave for the sake of simplicity.

	The seismic tomography aims to solve an inverse problem which minimizes the misfit functional
	\begin{equation}\label{eq:misfit}
		\mis(\rho) := \f12\int_0^T\rd t\int_\Og\rd x\, w(x)\left(\mcl{A}[u_\obs(t,x)] - \mcl{A}[u(t,x;\rho)]\right)^2, 
	\end{equation}
	where 
	$[0,\,T]$ is a fixed time window, $\Omega$ is the region of interest, 
	$w(x)$ represents the distribution of receivers 
	(\eg, in a finite-receiver setup, $\dpm w(x)=\frac1{N_{\rec}}\sum_{r=1}^{N_{\rec}}\delta_d(x-x_r)$, where $N_{\rec}$ is the number of receivers and $x_r$ is the location of the $r$-th receiver station), 
	$u_\obs(t,x)$ is the observed signal and $u(t,x)$ is the synthetic signal satisfies the forward propagating wave equation \eqref{eqn:wave3d}. We use $\mcl{A}$ to denote an observation operator to extract useful information from the signals, {\it e.g.}, in the  full-waveform inversion (FWI), $\mcl{A}[u]=u$ means all the information containing in a signal $u$ is used; and in the travel-time inversion, $\mcl{A}[u]$ is the time spent by a signal $u$ generated from a earthquake location $x_s$ propagated to a seismic receiver $x_r$.

	In order to solve the minimization problem, one needs to compute the Fr\'echet derivative of the misfit functional $\dt \mis/\dt\rho$. Without loss of generality, we take FWI as an example and compute ({\it c.f.} \cite{chai2018tomo})
	\begin{align*}
		\dt\mis&=-\int_0^T \rd t\int_\Og\rd x\,w(x)\,(u_{\obs}(t,x)-u(t,x))\,\dt u(t,x) 
		\\
		&=\int_0^T\rd t\int_\Og\rd x\int_0^t\rd\tau\int_\Og\rd y\, w(x)\,[u_{\obs}-u](t,x)\, G(t,x;\tau,y)\,\p_t^2{u}(\tau,y)\,\dt\rho(y) 
		\\
		&=\int_\Og\rd x\int_\Og\rd y\int_0^T\rd\tau\int_0^{T-\tau}\rd t\, w(x)\,[u_{\obs}-u](T-t,x)\, G(T-t,x;\tau,y)\,\p_t^2{u}(\tau,y)\,\dt\rho(y) 
		\\
		&=\int_\Og\rd y\int_0^T\rd\tau\int_0^{T-\tau}\rd t\int_\Og\rd x\, w(x)\,[u_{\obs}-u](T-t,x)\, G(T-\tau,y;t,x)\,\p_t^2{u}(\tau,y)\,\dt\rho(y) 
		\\
		&=\int\rd y\int_0^T\rd\tau \,\dt{\rho}(y)\,u^\dagger(T-\tau,y)\,\p_t^2{u}(\tau,y)
		\,,
	\end{align*}
	where the Green's function $G=G(t,x;\tau,y)$ solves
	\begin{equation*}
		\rho(x)\,\p_t^2G-\Dt_{x}G=\dt_d(t-\tau,x-y),
	\end{equation*}
	and $u^\dagger$ solves the adjoint wave equation 
	\eqref{eqn:adjoint} 
	\begin{align}\label{eqn:adjoint}
	\begin{split}
		&\rho(x)\,\p_t^2u^\dagger-\Dt_{x} u^\dagger=s^\dagger(t,x), ~\quad x\in\R^3, ~t>0,
		\\
		&u^\dagger(t=0,x) = 0, \quad \p_t u^\dagger(t=0,x) = 0,
	\end{split}
	\end{align}
	with the adjoint source function 
	\begin{equation}
		s^\dagger(t,x)=w(x)\,[u_\obs-u](T-t,x).
	\end{equation}
	Define the sensitivity kernel 
	\begin{equation}\label{eq:kernel}
		\begin{aligned}
			\K(x;\rho) :=& \int_0^T\rd t \, \rho(x)\,u^\dagger(T-t,x)\,\p_t^2{u}(t,x)
			= \int_0^T\rd t \, \rho(x)\,\p_tu^\dagger(T-t,x)\,\p_t{u}(t,x),
		\end{aligned}
	\end{equation}
	then one gets
	\begin{equation}\label{eqn:var_derivative}
		\dt\mis = \int_\Og\rd x \, K(x;\rho)\,\dt\log\rho(x).
	\end{equation}
	\begin{remark}
		The above computation can also performed to wave-equation-based travel-time inversion, yielding a similar formulations except that the adjoint source function becomes ({\it c.f.} \cite{chai2018tomo})
		\begin{equation*}
			s^\dagger(t,x)=w(x)g(t)\p_tu(t,x),
		\end{equation*}
		where $g(t)$ is a window function supported in $[0,\, T]$. 
	\end{remark}
	
	{
		After computing the sensitivity kernel \eqref{eq:kernel} and the Fr\'echet derivative of the misfit functional \eqref{eqn:var_derivative}, one can apply optimization methods to find the minimizer of \eqref{eq:misfit}, producing the desired velocity model by the relation $c=1/\sqrt\rho$. Classical methods include but not limited to gradient descent method, conjugate gradient method, Newtonian and qusi-Newtonian methods.  
		In this paper, we choose the gradient descent method for its simplicity, bringing convenience for derivation and proofs. We remark that the idea of reconstructing the gradient using random batch method can be used essentially the same way in other kinds of gradient-based optimization methods, at least for the purpose of numerical computing.
	}
	
	
	The gradient descent method can be formulated by
	\begin{equation}
		\frac{\rd \X}{\rd s} = -\nabla_\X\mis , 
	\end{equation}
	where $\X:=\log\rho$ and $\nabla_\X\mis := \K(x;\rho)$. 
	
	Given $\rho>\rho_0>0$, the map $\rho\rightarrow\X$ is one-to-one. 
	Therefore, we shall use $\X,\,\Y,\,\tX$ to denote density (or velocity) models all through this paper.
	Note that $X$ is a function of spatial variable $x\in\R^3$ and the iteration index $s\in\R^+$. Let's define
	$|\X(s)|^q := \int_\Omega |\X(x,s)|^q \rd x$, and 
	$\|\X(s)\| := \left(\E|\X(s)|^2\right)^{1/2}$. 
	For the rest part of the paper, we shall not write the dependence of $\X$ on the spatial variable $x$ explicitly but use $\X=\X(s)$ to put more focus on the iteration procedure. We also write $\K(\X)=\K(x;\rho)$.
	
	To make the inverse problem well-posed, a regularization term is added to the misfit functional \eqref{eq:misfit} and the above gradient flow is modified by
	\begin{equation}\label{eq:gradient_flow}
		\frac{\rd \X}{\rd s} = -K(\X) - \nabla_\X V(\X),
	\end{equation}
	where $V$ is a given regularization potential does not rely on solving the wave equations and we  assume that $V$ is strongly convex in $\X$ so that $V(\X)-\f{r}2\X^2$ is convex for some $r>0$, and $\nabla_{\X}V,\,\nabla_{\X}^2V$ have polynomial growth. 
	We remark here that the assumptions on $V$ is only for technical use in order to prove the convergence in the following section. In practice, these assumptions may be removed. It will be seen that, in Section \ref{sec:example}, the numerical examples all simply take $V\equiv0$ and we can still get numerical convergence results.

	\section{Frozen Gaussian approximation and random batch method}\label{sec:randombatch}
	In this section, we systematically introduce the construction of stochastic gradient by frozen Gaussian approximation (FGA) and random batch methods. We first describe how to use FGA to construct the gradient, and then use random batch method to construct the stochastic gradient.

	\subsection{FGA-based gradient construction}
	The sensitivity kernel $\K$ defined in \eqref{eq:kernel} is a cross-correlation of the forward and adjoint wavefields. For the convenience of latter discussion, we write $\K=\K(\X)$ to indicate dependence on the velocity model (noticing that both $u$ and $u^\dagger$ depend on $\X$ since they are synthetic solutions of \eqref{eqn:wave3d} and \eqref{eqn:adjoint}).  
	
	FGA approximates the wavefields by ({\it c.f.} \cite{chai2017frozen})
	\begin{equation}
		\label{eqn:FGA}
			u(t,x;\X)=\f1{N}\sum_{j=1}^N G_j(t,x;\X), 
			\quad{\text{and}}\quad
			u^\dagger(t,x;\X)=\f1{N}\sum_{j=1}^N G^\dagger_j(t,x;\X),
	\end{equation}
	where $G_j$ and $G^\dagger_j$ are Gaussian functions in the form of, \eg,
	\begin{equation*}
	    A\exp\left(\f{\ri}{\veps}P\cdot(x-Q)-\f1{2\veps}\abs{x-Q}^2\right)
	\end{equation*}
	with $A$, $Q$, and $P$ are functions of $(t,q,p)$
	determined by a set of ordinary differential equations(ODEs):
	\begin{equation}\label{eqn:FGAodes}
	\begin{dcases}
    \frac{\rd {Q}}{\rd t} = \partial_{P} H,
    &{Q}(0, {q}, {p}) = {q}, \\
    \frac{\rd {P}}{\rd t} = -\partial_{Q} H,
    &{P}(0, {q}, {p}) = {p},\\
    \frac{\rd A}{\rd t} = A\,\frac{\partial_{P} H \cdot \partial_{Q} H}{H}
    +\frac{A}{2}\,\tr\left(Z^{-1}\frac{\rd Z}{\rd t} \right),
    & A(0, {q}, {p}) = A_0(q,p),
    \end{dcases}
    \end{equation}
    with $H({Q},{P})=\pm c({Q})\abs{P}$, and the shorthand notations $\partial_{z}=\partial_{q}-\ri\partial_{p}$ and $Z=\partial_{z}(Q+\ri{P})$;
	see more details in \cite{chai2017frozen,chai2018tomo,hate2018fga}. 
	In \eqref{eqn:FGA} we assume the same beam number $N$ for both forward and adjoint simulations and $N$ does not change during the iteration procedure \eqref{eq:gradient_flow}. For the convenience of notations and latter discussion, we introduce 
	\begin{equation}
	    \label{eq:FGA_derivative}
	    \dG_j(t,x;\X):=\p_tG_j(t,x;\X),
	    \quad\text{and}\quad
	    \dG_j^\dagger(t,x;\X):=\rho(x)\,\p_tG_j^\dagger(t,x;\X).
	\end{equation}
	Then we approximate the kernel
	\begin{equation}\label{eq:k_full}
		\K(\X) = \sum_{k=1}^{T/\tau}\tau\,u^\dagger_k\,u_k
	\end{equation}
	where $u_k=\p_tu(t_k,x)$ and $u^\dagger_k=\rho(x)\,\p_tu^\dagger(T-t_k,x)$, $t_k=k\tau$, $k=1,2,\cdots,T/\tau$.
	The velocity model follows
	\begin{equation}\label{eq:gradient_flow_full}
		\frac{\rd\X}{\rd s} =  - \nabla_\X V(\X)  -\K(\X) ,
	\end{equation}
	at each iteration step $s\in[s_{m-1},s_m)$.

	\subsection{Stochastic gradient by random batch method}\label{sec:GRwRB}
	
	At each iteration step $s_m=mh$ and each time $t_k=k\tau$, 
	we randomly choose index sets $\BatG_{m,k},\,\BatG^\dagger_{m,k}\subset\{1,2,\cdots,N\}$,
	such that: 
	first, the number of indices $|\BatG_{m,k}|=|\BatG^\dagger_{m,k}|=p\ll N$;
	and second, all $\BatG_{m,k},\,\BatG^\dagger_{m,k}$ for $m=0,1,2,\cdots$ and $k=1,2,\cdots,T/\tau$ are independently identical distributed, 
	and the probability of an index belongs to a index set $\Pb(j\in\BatG_{m,k})=\Pb(j\in\BatG^\dagger_{m,k})=p/N$ for $1\leq j\leq N$.
	We call such $\BatG_{m,k}$ and $\BatG^\dagger_{m,k}$ as {\it random batches}.
	
	We first approximate the wavefields by random batch method
	\begin{align}
	\label{eqn:FGA_rb}
			\tu(t,x;\X)=\f1{p}\sum_{j\in\BatG_{m,k}} G_j(t,x;\X), 
			\quad\text{and}\quad
			\tua(t,x;\X)=\f1{p}\sum_{j\in\BatG^\dagger_{m,k}} G^\dagger_j(t,x;\X).
	\end{align}
	Then we define the stochastic gradient as
	\begin{equation}\label{eq:k_rb}
		\tK(\X) = \sum_{k=1}^{T/\tau}\tau\,\tua_k\,{\tu_k},
	\end{equation}
	where $\tu_k=\p_t\tu(t_k,x)$ and $\tua_k=\rho(x)\,\p_t\tua(T-t_k,x)$. 
	
	We call $\tK(\X)$ as the random batch kernel, and update the velocity model by 
	\begin{equation}\label{eq:gradient_flow_rb}
		\frac{\rd\X}{\rd s} =  - \nabla_\X V(\X) - \tK(\X) ,
	\end{equation}
	for each iteration step $s\in[s_{m-1},s_m)$.
	
	\begin{remark}\label{rem:idea}
		The main idea here is to use random batch summation to construct randomized wavefields and the corresponding sensitivity kernels. 
		It does not rely on the usage of FGA and one can replace it by any other efficient wave propagation solvers, {\it e.g.}, Gaussian beam method and specitral element method. 
	\end{remark}

	\begin{remark}
        As discussed in \cite{chai2017frozen}, the ODEs \eqref{eqn:FGAodes} can be embarrassingly parallelly computed, but the parallelization on computing the summation \eqref{eqn:FGA} is less efficient and technically involved. Therefore, the random batch summation \eqref{eqn:FGA_rb} can significantly reduce the workload of reconstructing wavefields and further improve the efficiency of the FGA computation. 
	\end{remark}

	\begin{remark}\label{rem:iniFGA}
	    To get a well-resolved wavefield, the number of beams $N$ is typically at the order of $\veps^{-d/2}$ where $\veps$ is proportional to the wave-length and $d=2,3$ is the dimensionality of the space, thus $N$ is usually a large number when doing a high-frequency simulation where $\veps\ll1$. Recently, a random sampling method \cite{Xie2021} was developed for the Schr\"odinger equation, which can reduce the number of beams significantly by choosing beams via a preliminary distribution. This sampling idea is restricted to Gaussian or WKB initial data for which one can determine whether a beam is more ``important" or less. 
	    However, for wave equations with Dirac delta sources, the beams are almost uniformly distributed, and thus it's not clear yet how to find a good sampling strategy. 
	\end{remark}

	As in Algorithm \ref{sgdrb}, we present a brief pseudo-code for implementation of the procedure of the proposed stochastic gradient descent by random batch method. We remark that if in line \ref{gen_rb} and \ref{gen_rba} we set $\BatG_{m,k}$ and $\BatG^\dagger_{m,k}$ to be the whole index set $\{1,2,\cdots,N\}$, the algorithm recover the classical gradient descent method.
	\begin{algorithm}
		\caption{Stochastic gradient descent by random batch method}
		\label{sgdrb}
		\begin{algorithmic}[1] 
        \Procedure{Main loop}{} 
 			\State $\X=\X_0$, $m=0$
			\State Compute $\mis(\X)$
			\While{{$m<M^*$}} \Comment{ $M^*$ is a given number indicate the maximum iteration steps }
			\State $\tK$~$\leftarrow$~\Call{Random Batch Kernel}{} \Comment{Compute the gradient by random batch method}
			\State $\X$~$\leftarrow$~$\X-\alpha(\tK+\nabla_\X V)$, $m$~$\leftarrow$~$m+1$
			\State Compute $\mis(\X)$
			\State \textbf{if {$(\mis(\X)<\mis^*)$}} \textbf{exit} \Comment{If the misfit smaller than a given threshold $J^*$, exit iteration loop}
			\EndWhile
        \EndProcedure
			\Procedure{Random Batch Kernel} {$m$} 
			\State Given discrete source and receiver locations $x_s$ and $x_r$
			\State Initialize $G_j(t,x;\X)$ at $t=0$ for all $j=1,\cdots,N$ 
			\For {$k=1:T/\tau$} \Comment{TIME EVOLUTION FOR FORWARD SIMULATION}
			\State $G_j(t,x;\X)$~$\leftarrow$~$G_j(t+\tau)$ for all $j=1,\cdots,N$ \Comment{Evolve FGA ODEs for one time step} 
			\State Generate independent random batch $\BatG_{m,k}$ \label{gen_rb}
			\State 
			Compute $\tu(t,x;\X)$ by equation \eqref{eqn:FGA_rb}; 
			Compute $u(t,x_r;\X)$ by equation \eqref{eqn:FGA}
			\EndFor
			\State Initialize $G_j^\dagger(t,x;\X)$ at $t=0$ for all $j=1,\cdots,N$
			\For {$k=1:T/\tau$} \Comment{TIME EVOLUTION FOR ADJOINT SIMULATION}
			\State $G_j^\dagger(t,x;\X)$~$\leftarrow$~$G_j^\dagger(t+\tau)$ for all $j=1,\cdots,N$ \Comment{Evolve FGA ODEs for one time step} 
			\State Generate independent random batch $\BatG^\dagger_{m,k}$ \label{gen_rba}
			\State Compute $\tua(t,x;\X)$ by equation  \eqref{eqn:FGA_rb}
			\EndFor
			\State Compute $\tK$ by equation \eqref{eq:k_rb}
			\EndProcedure ~ \Return{$\tK$ and $u(\cdot,x_r;\X)$}
		\end{algorithmic}
		
	\end{algorithm}
	
	\subsection{Preliminary results}\label{sec:preliminary} In this subsection, we prove a few lemmas, as a preparation to prove the convergence theorem in the next subsection. 

	First, let's state the Lipschitz continuity properties of $\K$ and $\tK$.
	\begin{prop}
	The kernels defined in \eqref{eq:k_full} and \eqref{eq:k_rb} are Lipschitz continuous in $\X$.
	\end{prop}
	\begin{proof}
	As it defined in equations \eqref{eqn:FGA}-\eqref{eq:k_full}, the kernel $K$ can be seen as a summation of $\dG_j\dG^\dagger_l$'s. Each $G_j$ or $G^\dagger_l$ is a Gaussian function whose parameters are given by a set of ODEs \eqref{eqn:FGAodes}. By the smooth dependence on the initial condition and parameters for solution of ODEs, one can deduce that $\dG_j$ and $\dG^\dagger_l$'s are smooth in $\X$ and thus $\K$ is Lipschitz continuous in $\X$. Similarly $\tK$ is Lipschitz continous in $\X$.
	\end{proof}

	Note that equation \eqref{eq:gradient_flow_rb} can be rewritten as
	\begin{equation}\label{eq:gradient_flow_rb_chi}
		\frac{\rd\X}{\rd s} =  - \nabla_\X V(\X) -{\K}(\X(s)) - \chi_m(\X(s)),
	\end{equation}
	where
	\begin{equation}\label{eq:chi}
		\chi_m(\X):={\tK}(\X)-\K(\X).
	\end{equation}
	Thus to analyze the convergence of the random batch method, the key is a precise estimate on $\chi_m$, for which we have the following lemma.
	\begin{lem}\label{lemma:var}
		Let $\Y$ be a velocity model, fixed and determined, and $\K(\Y)$ and $\tK(\Y)$ defined as in \eqref{eq:k_full} and \eqref{eq:k_rb}, respectively. Then
		\begin{align}
			\label{eqn:exp}
			\E[\chi_m(\Y)] & = 0, 
			\\
			\label{eqn:var}
			\E[\chi_m^2(\Y)] & = \left(\f1{p}-\f1{N}\right) \tau \Lambda,
		\end{align}
		where
		\begin{equation}
			\Lambda 
			= \f\tau{N-1} \sum_{k=1}^{T/\tau} \sum_{j=1}^N \left[
			\left(\dG_j\,-\,\f1{N}\sum_{l=1}^N\dG_l\right)^2 \E(\tua_k)^2 
			+ u_k^2  \left(\dG^\dagger_j\,-\,\f1{N}\sum_{l=1}^N\dG^\dagger_l\right)^2 
			\right].
		\end{equation}
	\end{lem}
	\begin{proof}
		The expectation \eqref{eqn:exp} is straightforward, and we only prove the variance equality \eqref{eqn:var}.
		Noticing that $\tu_k$, $\tua_j$ for $k,j=1,\cdots,\,T/\tau$ are independent, then
		\begin{align}\label{eqn:var_1st}
		\begin{split}
			\E\chi_m^2 = \tau^2\sum_{k=1}^{T/\tau}\E\left[(u_ku^\dagger_k-\tu_k\tua_k)^2\right]
			&= \tau^2\sum_{k=1}^{T/\tau}\left( 
			\E (\tu_k)^2\,\E (\tua_k)^2 - u_k^2(u^\dagger_k)^2
			\right)
			\\
			&=\tau^2\sum_{k=1}^{T/\tau}\left( 
			\left(\E(\tu_k)^2 - u_k^2\right)\E(\tua_k)^2 
			+ u_k^2\left(\E(\tua_k)^2 - (u^\dagger_k)^2\right) 
			\right),
		\end{split}
		\end{align}
		where $u_k,\,\tu_k,\,u_k^\dagger$ and $\tua_k$ takes the form of
		\begin{equation*}
			u_k  = \f1N\sum_{j=1}^N\dG_j , \quad
			\tu_k = \f1p\sum_{j\in\BatG_{m,k}}\dG_j ,
			\quad
			u_k^\dagger  = \f1N\sum_{j=1}^N\dG_j^\dagger , \quad
			\tua_k  = \f1p\sum_{j\in\BatG^\dagger_{m,k}}\dG^\dagger_j .
		\end{equation*}
		Then one can compute
		\begin{align*}
			\E(\tu_k)^2 & = \f1{p^2}\sum_{j=1}^N \dG_j^2 \, \Pb(j\in\BatG_{m,k}) 
			\,+\,  \f1{p^2}\sum_{j,l:j\neq l} \dG_j\,\dG_l \, \Pb \left( j\in\BatG_{m,k} \text{ and } l\in\BatG_{m,k} \right)  
			\\
			& = \f1{pN}\sum_{j=1}^N \dG_j^2 
			\,+\,  \f{p-1}{pN(N-1)}\sum_{j,l:j\neq l} \dG_j\,\dG_l ,
		\end{align*}	
		and thus
		\begin{align*}
			\E(\tu_k)^2 - u_k^2 
			& = \left(\f1{pN}-\f1{N^2}\right)\sum_{j=1}^N \dG_j^2 
			\,+\,  \left(\f{p-1}{pN(N-1)}-\f1{N^2}\right)\sum_{j,l:j\neq l} \dG_j\,\dG_l 
			\\
			& =  \left(\f1{p}-\f1{N}\right)\left(\f1N\sum_{j=1}^N \dG_j^2 
			\,-\,  \f1{N(N-1)}\sum_{j,l:j\neq l} \dG_j\,\dG_l\right)
			\\
			& = \left(\f1{p}-\f1{N}\right)\f1{N-1}\sum_{j=1}^N \left(\dG_j \,-\,\f1{N}\sum_{l=1}^N \dG_l\right)^2.
		\end{align*}
		The $\E(\tua_k)^2 - (u^\dagger_k)^2$ can be compute in a analog way, and then we can obtain \eqref{eqn:var}.
	\end{proof}
	
	Let $\X(s)$ be solution to \eqref{eq:gradient_flow_full} and $\tX(s)$ be solution to \eqref{eq:gradient_flow_rb}. Define $Z(s):=\tX(s)-\X(s)$, and let $\mcl{F}_{m-1}$ be $\sigma$-algebra generated by the random batch construction for $s\leq s_{m-1}$. 
	The following lemmas are devoted to the stability and truncation error analysis of the random batch method.
	\begin{lem} One can have the following estimate for $\X$ and $\tX$
		\begin{equation}
			\sup_{t>0}\left(|\X|^q+\E|\tX|^q\right) \leq C_q.
		\end{equation}
	\end{lem}
	\begin{proof}
		\begin{equation*}
			\dds{|\X|^q} 
			= |\X|^{q-2}\X\cdot\dds{\X} 
			= -|\X|^{q-2}\X\cdot \left(\nabla_\X V(\X) + \K(\X)\right)
		\end{equation*}
		\begin{equation*}
			\X\cdot\nabla_\X V(\X)
			= (\X - 0)\cdot(\nabla_\X V(\X) -\nabla_\X V(0) ) + \X\cdot\nabla_\X V(0)
			= (\X - 0)^2:\nabla_\X^2 V(\X^*)  + \X\cdot\nabla_\X V(0) 
		\end{equation*}
		\begin{equation*}
			\dds{|\X|^q} 
			\leq -qr|\X|^q + \|\K\|_\infty |\X|^{q-1} 
			\leq -qr|\X|^q + \|\K\|_\infty \left(\f{q-1}{q}\nu|\X|^{q}+\f1{q\nu^{q-1}}\right).
		\end{equation*}
		Thus $|\X|^q\leq C_q$. Similarly, $\E|\tX|^q\leq C_q$.
	\end{proof}

	\begin{lem}
		For $s\in[s_{m-1},s)$, it holds that
		\begin{equation}\label{eqn:X-inc}
			\left\|\tX(s)-\tX(s_{m-1})\right\| \leq Ch .
		\end{equation}
	\end{lem}
	\begin{proof}
		Direct computation shows
		\begin{equation*}
			\dds{}\left\|\tX(s)-\tX(s_{m-1})\right\|^2
			= 
			-2\,\E\left[\left(\tX(s)-\tX(s_{m-1})\right)\,\left(\nabla_{\X}V\left(\tX(s)\right)+\tK\left(\tX(s)\right)\right)\right],
		\end{equation*}
		then by H\"older's inequality, one has
		\begin{equation*}
			\dds{}\left\|\tX(s)-\tX(s_{m-1})\right\|^2
			\leq 
			C\left\|\tX(s)-\tX(s_{m-1})\right\|\left(\left\|\nabla_{\X}V\left(\tX(s)\right)\right\|+\left\|\tK\left(\tX(s)\right)\right\|\right).
		\end{equation*}
		Note that $\nabla_\X V \leq C(1+|\X|^q)$ for some $q$ and thus $\|\nabla_{\X}V\|$ is bounded; and $\tK$ is a cross-corelation of two wavefield constructed from Gaussians where the Gaussians are determined by a set of ODEs depending on the velocity model $\tX$ smoothly, so Gaussians are bounded and so is $\|\tK\|$. Thus
		\begin{equation*}
			\dds{}\left\|\tX(s)-\tX(s_{m-1})\right\|^2
			\leq 
			C\left\|\tX(s)-\tX(s_{m-1})\right\|,
		\end{equation*}
		and then the estimate \eqref{eqn:X-inc} follows.
	\end{proof}

	\begin{lem}\label{lemma:Z_inc}
		For $s\in[s_{m-1},s)$,
		\begin{equation}
			\left\|Z(s)-Z(s_{m-1})\right\| \leq Ch ,
		\end{equation}
		and
		\begin{equation}
			\E\left|\left(Z(s)-Z(s_{m-1})\right)\,\chi_m\left(\X(s)\right)\right| 
			\leq Ch\left[\left(\|Z(s)\| + \|Z(s)\|^2\right) + h\right]
			+ \f{h\tau}{p}  \|\Lambda\|_\infty.
		\end{equation}
	\end{lem}
	\begin{proof}
		\begin{equation*}
			\dds{Z} 
			=  - \nabla_\X V(\tX) + \nabla_\X V(\X) - \tK(\tX) + \K(\X) 
		\end{equation*}
		thus
		\begin{equation*}
			\f12\dds{Z^2}  \leq\,  - (r-L)Z^2, 
		\end{equation*}
		which implies that, for $s\in[s_{m-1},s_m)$, one has
		\begin{equation*}
			|Z(s)|\leq|Z(s_{m-1})|+Ch , 
			\quad\text{and }
			\|Z(s)-Z(s_{m-1})\|\leq Ch .
		\end{equation*}
		\begin{equation*}
			-\dds{Z} 
			= \nabla_\X V(\tX) - \nabla_\X V(\X)  + \tK(\tX) - \tK(\X) + \chi_m(\X)
		\end{equation*}
		Since
		\begin{equation*}
			\left|\nabla_\X V(\tX) - \nabla_\X V(\X)\right|
			\leq \left|(\tX - \X)\cdot\nabla_{\X}^2V(\X^*)\right|,
		\end{equation*}
		then
		\begin{align*}
			&\E\left|\left(\nabla_\X V(\tX(s')) - \nabla_\X V(\X(s'))\right)\chi_m(\X(s))\right|
			\\
			&\leq \left\|\chi_m(\X(s))\right\|_\infty\left\|(\tX(s') - \X(s'))\right\|
			\left(\E\left[|\tX(s')|^{q_1}+|\X(s')|^{q_1}\right]^2\right)^{1/2}
			\leq C \|Z(s')\|.
		\end{align*}
		Therefore
		\begin{align*}
			&\E\left|\left(Z(s)-Z(s_{m-1})\right)\,\chi_m\left(\X(s)\right)\right| 
			\\
			\leq\,& 
			\int_{s_{m-1}}^s \rd s'\,
			\left\{ 
			C\,\|Z(s')\|
			+ 
			\E\left|\chi_m\left(\X(s')\right)\,\chi_m\left(\X(s)\right)\right|
			+ 
			\E\left|\left(\tK(\tX(s'))-\tK(\X(s'))\right)\,\chi_m\left(\X(s)\right)\right|
			\right\}
			\\
			\leq\,& 
			\int_{s_{m-1}}^s \rd s'\,
			\left\{	
			C\,\|Z(s')\|
			+ \E\left[\chi_m\left(\X(s)\right)^2\right] 
			+ \f12\,\E\left[\chi_m\left(\X(s')\right)^2\right] 
			+ \f12\, \E\left[\left(\tK(\tX(s'))-\tK(\X(s'))\right)^2\right]
			\right\}
		\end{align*}
		The second and third terms are controlled by Lemma \ref{lemma:var} since $\X$ is independent of the random batch, and thus
		\begin{equation*}
			\E\left[\chi_m\left(\X(s)\right)^2+\chi_m\left(\X(s)\right)^2\right] 
			\leq C \left(\f1{p}-\f1{N}\right)\tau  \|\Lambda\|_\infty.
		\end{equation*}
		The fourth term is controlled by using the Lipschitz continuity of $\tK$:
		\begin{equation*}
			\E\left[\left(\tK(\tX)-\tK(\X)\right)^2\right]
			\leq L^2\,\E\left[\left(\tX-\X\right)^2\right]
			=L^2\|Z\|^2.
		\end{equation*}
		Then
		\begin{equation*}
			\E\left|\left(Z(s)-Z(s_{m-1})\right)\,\chi_m\left(\X(s)\right)\right| 
			\leq
			C\left[\left(\|Z(s)\| + \|Z(s)\|^2\right)h + h^2\right]
			+ \left(\f1{p}-\f1{N}\right)h\tau  \|\Lambda\|_\infty .
		\end{equation*}
	\end{proof}

	\subsection{Main theorem}\label{sec:mainthm}
	In this subsection, we present the main convergence theorem. As in Remark~\ref{rem:idea}, the main idea of the proposed method is to use random batch method for the construction of randomized wavefields and the corresponding sensitivity kernels, and one can replace FGA by any other efficient wave propagation solvers. Therefore, we shall only focus the convergence of random batch method to the deterministic gradient decent method, assuming that the chosen wave propagation solvers can provide convergent numerical results. We refer to \cite{LuYa:CPAM, chai2021fgaconv} for the convergent results of the FGA solvers. 
	
	\begin{thm}\label{thm:1}
		As the iteration step $h$ goes to zero, $\tX$ converges to $\X$  in the mean square sense. More precisely, we have the following estimate
		\begin{equation}
			\sup_{s\ge0}\|Z(s)\| \leq  C\sqrt{\f{h\tau}{p} + Ch^2}.
		\end{equation}
	\end{thm}
	\begin{proof}
		\begin{align*}
			\f12\dds{}\E Z^2
			=& - \E\left[ Z(s) \, \left(  \nabla_\X V (\tX(s)) - \nabla_\X V \left(\X(s)\right)  +  \K(\tX(s)) - \K\left(\X(s)\right) \right) \right]
			-\E\left[Z(s)\,\chi_m\left(\tX(s)\right)\right]
			\\
			\leq\, &
			-(r-L)\E Z^2 	 - \E\left[Z(s)\,\chi_m\left(\tX(s)\right)\right] .
		\end{align*}
		Let
		\begin{equation*}
			R(s) := \E\left[Z(s)\,\chi_m\left(\tX(s)\right)\right].
		\end{equation*}
		\begin{align*}
			R(s) =&~~ 
			\E\left[Z(s_{m-1})\,\chi_m\left(\tX(s_{m-1})\right)\right]
			\\&+\E\left[Z(s_{m-1})\,\left(\chi_m\left(\tX(s)\right) - \chi_m\left(\tX(s_{m-1})\right)\right)\right]
			\\&+\E\left[\left(Z(s)-Z(s_{m-1})\right)\,\chi_m\left(\X(s)\right)\right]
			\\&+\E\left[\left(Z(s)-Z(s_{m-1})\right)\,\left(\chi_m\left(\tX(s)\right) - \chi_m\left(\X(s)\right)\right)\right]
			\\
			=&:\,I_1+I_2+I_3+I_4
		\end{align*}
		For the first term,
		\begin{equation*}
			I_1 
			= \E\left[\E\left[\left.Z(s_{m-1})\,\chi_m\left(\tX(s_{m-1})\right)\right|\mcl{F}_{m-1}\right]\right]
			= \E\left[Z(s_{m-1})\,\E\left[\left.\chi_m\left(\tX(s_{m-1})\right)\right|\mcl{F}_{m-1}\right]\right]
			= 0 .
		\end{equation*}
		For the second term,
		\begin{align*}
			I_2
			&=\E\left[Z(s_{m-1})\,\left(\chi_m\left(\tX(s)\right) - \chi_m\left(\tX(s_{m-1})\right)\right)\right]
			\\
			&\leq C\left\|Z(s_{m-1})\right\|\left\|\chi_m\left(\tX(s)\right) - \chi_m\left(\tX(s_{m-1})\right)\right\|
			\leq 2LC\left\|Z(s_{m-1})\right\|\left\|\tX(s) - \tX(s_{m-1})\right\|
			\\
			&
			\leq
			C\|Z(s)\|h + Ch^2,
		\end{align*}
		where for the second inequality we have used the Lipschitz continuity of $K$ and $\tK$. 
		For the third term, by Lemma \ref{lemma:Z_inc}
		\begin{equation*}
			I_3\leq Ch\left[\left(\|Z(s)\| + \|Z(s)\|^2\right) + h\right]
			+ \f{h\tau}{p}  \|\Lambda\|_\infty.
		\end{equation*}
		For the fourth term,
		\begin{equation*}
			I_4\leq \left\|Z(s)-Z(s_{m-1})\right\|\,\left\|\chi_m\left(\tX(s)\right) - \chi_m\left(\X(s)\right)\right\|
			\leq C\|Z(s)\|h
			.
		\end{equation*}
		Hence,
		\begin{equation*}
			R(s) \leq C\|Z(s)\|h + C\f{h\tau}{p} + Ch^2,
		\end{equation*}
		and
		\begin{equation*}
			\dds{}\|Z\|^2 \leq -(r-L) \|Z\|^2 + Ch\|Z\| + C\f{h\tau}{p} + Ch^2,
		\end{equation*}
		which implies
		\begin{equation*}
			\sup_{s\ge0}\|Z(s)\|^2 \leq  C\f{h\tau}{p} + Ch^2.
		\end{equation*}
	\end{proof}

{

	\subsection{Random batch on source-receiver pairs} In this subsection, as a byproduct and by essentially following the same proof strategies in Sections~\ref{sec:preliminary} and \ref{sec:mainthm}, we provide convergence results for the accelerated full-waveform inversion method using dynamic mini-batches proposed in \cite{van2020accelerated}. The method uses spectral element methods to compute the synthetic sesmograms, and apply random batches on the source-receiver pairs. For convenience, we briefly review the method using consistent notations of this paper. Assume there are $N_\rec$ receiver stations and $N_\src$ earthquake events, one can rewrite the sensitivity kernel by
	\begin{align}
	    \K = \f1{N_\rec N_\src}\sum_{r=1}^{N_\rec}\sum_{s=1}^{N_\src}\K_{rs},
	    \quad\text{where}\quad
	    \K_{rs} = \int_0^T\rd t \, \rho(x)\,\p_tu^\dagger(T-t,x;x_r)\,\p_t{u}(t,x;x_s),
	\end{align}
	where one uses spectral element method to solve $u(\cdot,\cdot;x_s)$ by the wave equation \eqref{eqn:wave3d} with the source located at $x=x_s$, and 
	$u^\dagger(\cdot,\cdot;x_r)$ by the adjoint wave equation \eqref{eqn:adjoint} with the adjoint source function
	\begin{align*}
	    s^\dagger(t,x) = [u_\obs-u](T-t,x)\,\delta_d(x-x_r).
	\end{align*}
	To apply the random batch method, in each iteration step $m$ we choose a receiver index subset $\BatR_m\subset\{1,2,\cdots,N_\rec\}$ and a source index subset $\BatS_m\subset\{1,2,\cdots,N_\src\}$, randomly and independently. The random batch kernel is then given by
	\begin{align}
	    \tK=\f1{p_\rec p_\src}\sum_{r\in\BatR_m}\sum_{s\in\BatS_m}\K_{rs}.
	\end{align}
	Now one can use this kernel in the main loop of Algorithm \ref{sgdrb} to update the velocity model $\X$. 
	
	Our contribution here is to give a convergence result for this random batch method in source-receiver paris in analog to Theorem \ref{thm:1}. Since the strategy of the proof is essentially the same, we only state the following key lemma.
	\begin{lem}\label{lemma:varRec}
		Let $\Y$ be a velocity model, fixed and determined, and $\chi_m(\Y):={\tK}(\Y)-\K(\Y)$. Then it holds that $\E_\src[\chi_m(\Y)] = \E[\chi_m(\Y)]  = 0$ and
		\begin{align}
		    \label{eqn:varRec}
			\E_\src[\chi_m^2(\Y)] & = \left(\f1{p_\rec}-\f1{N_\rec}\right) \Lambda_\rec,
			\\
		    \label{eqn:varRecSrc}
			\E[\chi_m^2(\Y)] & = \left(\f1{p_\src}-\f1{N_\src}\right) \E\Lambda_\src + \E_\src[\chi_m^2(\Y)],
		\end{align}
		where we have used the short notation $\E_\src$ for conditional expectation $\E_\src[\,\cdot\,]:=\E[\,\cdot\,\left.\right|\,\BatS_m={1,\cdots,N_\obs}]$, and 
		\begin{align}
			\Lambda_\rec 
			= \f1{N_\rec-1} \sum_{r=1}^{N_\rec} 
			\left(\K_r\,-\,\f1{N_\rec}\sum_{l=1}^N\K_l\right)^2,
			&\quad\text{with}\quad\K_r=\f1{N_\src}\sum_{s=1}^{N_\src}\K_{rs},
			\\
			\Lambda_\src
			= \f1{N_\src-1} \sum_{r=1}^{N_\src} 
			\left(\tK_s\,-\,\f1{N_\src}\sum_{l=1}^N\tK_l\right)^2,
			&\quad\text{with}\quad\tK_s=\f1{p_\rec}\sum_{r\in\BatR_m}\K_{rs}.
		\end{align}
	\end{lem}
	\begin{proof}
	It is straightforward to show $\E_\src[\chi_m(\Y)] = \E[\chi_m(\Y)] = 0$.
	To show \eqref{eqn:varRec}, we compute
	\begin{align*}
	    \E_\src[\chi_m^2(\Y)] & = \E\left[\left(\f1{p_\rec}\sum_{r\in\BatR_m}\K_r-\f1{N_\rec}\sum_{r=1}^{N_\rec}\K_r\right)^2\right]
	    = \f1{p_\rec^2}\E\left(\sum_{r\in\BatR_m}\K_r\right)^2-\f1{N_\rec^2}\left(\sum_{r=1}^{N_\rec}\K_r\right)^2.
	\end{align*}
	Notice that
	\begin{align*}
	    \E\left(\sum_{r\in\BatR_m}\K_r\right)^2 & = 
	    \sum_{r=1}^{N_\rec} \K_r^2 \, \Pb(r\in\BatR_{m}) 
			\,+\, \sum_{r,l:r\neq l} \K_r\,\K_l \, \Pb \left( r\in\BatR_{m} \text{ and } l\in\BatR_{m} \right)  
			\\
			& = \f{p_\rec}{N_\rec}\sum_{r=1}^{N_\rec} \K_r^2 
			\,+\,  \f{p_\rec(p_\rec-1)}{N_\rec(N_\rec-1)}\sum_{r,l:r\neq l} \K_r\,\K_l ,
	\end{align*}
	and thus
	\begin{align*}
	    \E_\src[\chi_m^2(\Y)] 
	    & = \left(\f1{p_\rec N_\rec}-\f1{N_\rec^2}\right)\sum_{r=1}^{N_\rec} \K_r^2 
			\,+\,  \left(\f{p_\rec-1}{p_\rec\,N_\rec(N_\rec-1)}-\f1{N_\rec^2}\right)\sum_{r,l:r\neq l} \K_r\,\K_l ,
		\\
		& = \left(\f1{p_\rec}-\f1{N_\rec}\right)\f1{N_\rec-1} \sum_{r=1}^{N_\rec} 
			\left(\K_r\,-\,\f1{N_\rec}\sum_{l=1}^{N_\rec}\K_l\right)^2 
			,	
	\end{align*}
yielding \eqref{eqn:varRec}. Then \eqref{eqn:varRecSrc} can be obtained by taking the expectation with respect to $\BatS_m$ and $\BatR_m$ separately. 
	\end{proof}
}	
	
	\section{Numerical examples}\label{sec:example}
	In this section, we present some synthetic tomography tests using random batch gradient reconstruction, where the wave equations are solved by FGA and the sensitivity kernel are constructed using \eqref{eq:k_rb}. 
	{ 
	Remark that the random batch gradient reconstruction method we proposed can be applied to any gradient-based iterative method, but the proof of the convergence may be more complicated than that in the previous section. So in the following subsections, we mainly use L-BFGS for the iterations and show the convergence numerically but leave the rigorous proof for further studies. We present one gradient descent example in Section \ref{sec:exp_gd}.  
	We remark that all the computations are performed on a Dell T7920 workstation with dual Intel Xeon Gold 6130 Processor(16 Cores, 22M Cache, 2.10 GHz) and compiled with GFORTRAN and MPICH.
	}
	
	\subsection{Full-waveform inversion for a 2D model}\label{sec:fwi2d}
	In the first example, we present a test using full-waveform inversion(FWI) to image a 2D (in $x-z$ plane) square region. As a proof of methodology, we set point receivers on the top and right of the square region, and set point sources aligning on the bottom and left of the square region.
	The target velocity field is set as follows,
	\begin{equation}\label{eqn:2dnewmodel}
		c(x,z) = C_0\left(1-\alpha\exp\left(-\f\beta{L^2}\left((x-x_{c1})^2+(z-z_c)^2\right)\right)+\alpha\exp\left(-\f\beta{L^2}\left((x-x_{c2})^2+(z-z_c)^2\right)\right)\right),
	\end{equation}
	where $C_0=2500\cunit$, $x_{c1}=1344\xunit$, $x_{c2}=1824\xunit$, $z_c=L=1584\xunit$, $\alpha=0.03$, $\beta=24.2$. 
	See FIG. \ref{fig:2Dfwi1}(a) for a demonstration of the setup. FWI iteration starts with the background velocity, that is $c_0\equiv2500\cunit$ homogeneously. In this example, the beam number $N=32766$, and $\epsilon = L/256 $.  
	
	We use FGA to simulate the forward and adjoint wave equations. To reconstruct the wavefields and kernels, we use two strategies to generate random batch,
	\begin{itemize}
		\item strategy 1: as we proposed in Section \ref{sec:GRwRB}, for each iteration step $m$ and time evolution step $k$, we choose such that $\{\BatG_{m,k},\,\BatG_{m,k}^\dagger\,:\,m\in\N,\,k=0,1,\cdots,T/\tau\}$ is independent.
		\item strategy 2: for each iteration step $m$ we choose two batches $\BatG_{m}$ and $\BatG_{m}^\dagger$ independently, and set $\BatG_{m,k}=\BatG_{m}$, $\BatG_{m,k}^\dagger=\BatG_{m}^\dagger$ for all $,k=0,1,\cdots,T/\tau$, that is, we lose the independence for time evolution steps.
	\end{itemize}
	
	In FIG. \ref{fig:2Dfwi1}(a) and (d), we plot the resulted velocities after four iteration steps using batch strategy 1 with sampling rate $p/N = 10\%$ and 2.5\%, respectively,
	and one can see the low-velocity region has already been captured (though there are blurs and artifacts). We also plot time-shots of the wavefields for both $10\%$ and $2.5\%$ reconstruction in FIG. \ref{fig:2Dfwi1}(b) and (e), respectively, and the kernels for both $10\%$ and $2.5\%$ reconstruction in FIG. \ref{fig:2Dfwi1}(c) and (f), respectively. 
	\begin{figure}
		\label{fig:2Dfwi1}
		\centering
		\subfigure[]{
			\label{fig:2Dfwi_result_100}
			\begin{minipage}[t]{0.32\linewidth}
				\centering
				\includegraphics[width=\linewidth]{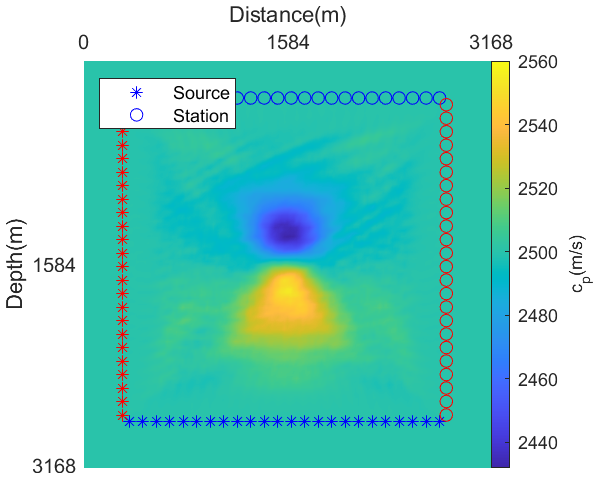}
			\end{minipage}%
		}%
		\subfigure[]{
			\label{fig:2Dfwi_wave_100}
			\begin{minipage}[t]{0.32\linewidth}
				\centering
				\includegraphics[width=\linewidth]{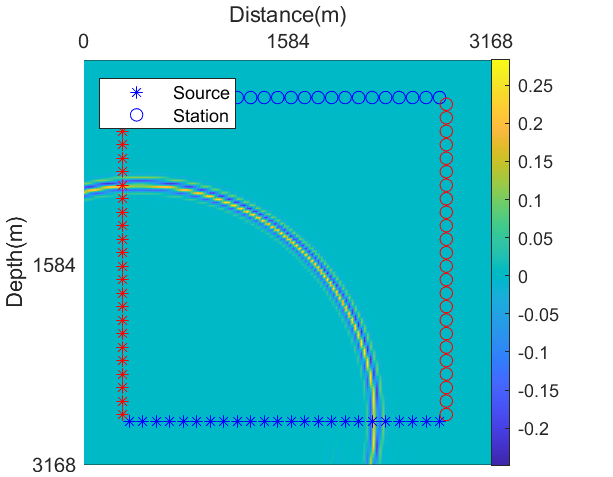} 
			\end{minipage}%
		}%
		\subfigure[]{
			\label{fig:2Dfwi_kernel_100}
			\begin{minipage}[t]{0.32\linewidth}
				\centering
				\includegraphics[width=\linewidth]{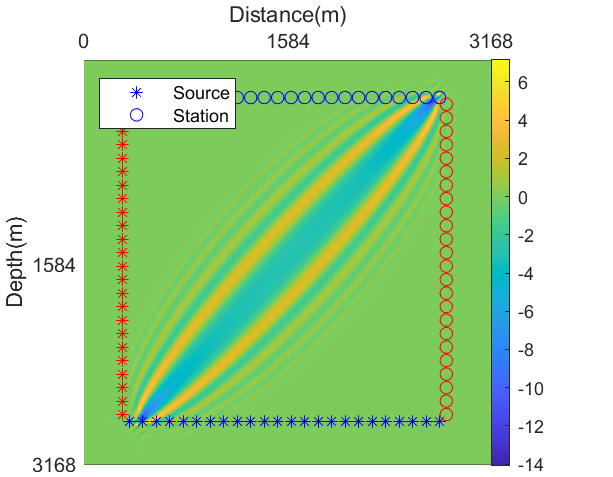}
			\end{minipage}%
		}%
		\\
		\subfigure[]{
			\label{fig:2Dfwi_result}
			\begin{minipage}[t]{0.32\linewidth}
				\centering
				\includegraphics[width=\linewidth]{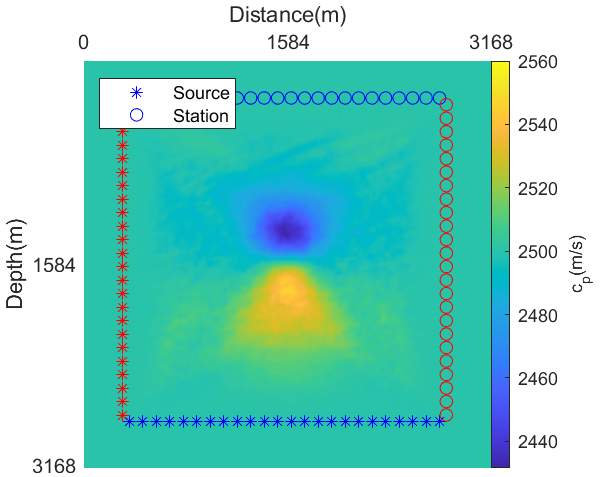}
			\end{minipage}%
		}%
		\subfigure[]{
			\label{fig:2Dfwi_wave_20}
			\begin{minipage}[t]{0.32\linewidth}
				\centering
				\includegraphics[width=\linewidth]{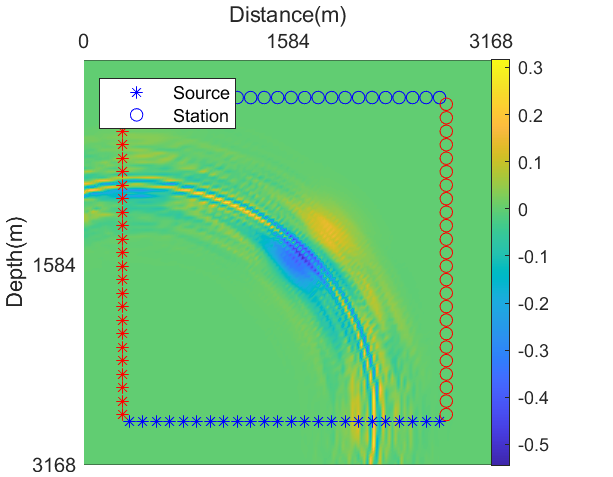} 
			\end{minipage}%
		}%
		\subfigure[]{
			\label{fig:2Dfwi_kernel_20}
			\begin{minipage}[t]{0.32\linewidth}
				\centering
				\includegraphics[width=\linewidth]{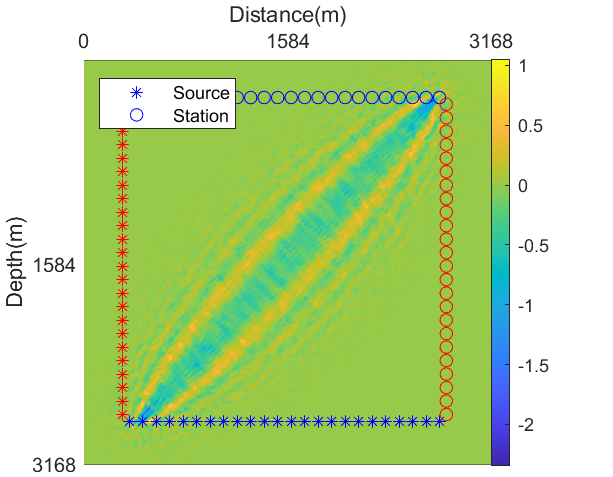} 
			\end{minipage}%
		}%
		\caption{Example \ref{sec:fwi2d}. 2D FWI results with batch strategy 1. (a) and (d) plot the resulted velocities after five and three iterations using batch strategy 1 with sampling rate $10\%$ and $2.5\%$, respectively. Comparisons of the same setup are also given for the wavefields ((b) and (e)) and kernels ((c) and (f)).}
	\end{figure}
	For a comparison, we use batch strategy 2 to generate batches and redo the test with the sampling rate $2.5\%$. In FIG. \ref{fig:2Dfwi2}, the inversion result is bad even the wavefield and kernel look ``okay". The decay of misfit functional for these two different strategies are shown in FIG. \ref{fig:2Dfwimis}.
	\begin{figure}
		\label{fig:2Dfwi2}
		\centering
		\subfigure[]{
			\label{fig:2Dfwi2_result}
			\begin{minipage}[t]{0.32\linewidth}
				\centering
				\includegraphics[width=\linewidth]{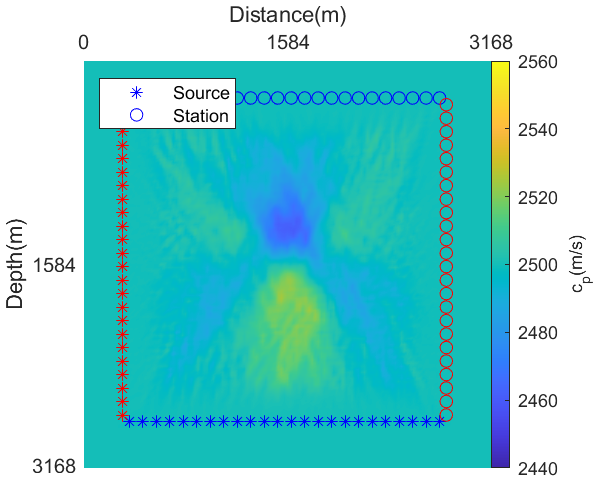}
			\end{minipage}%
		}%
		\subfigure[]{
			\label{fig:2Dfwi2_wave}
			\begin{minipage}[t]{0.32\linewidth}
				\centering
				\includegraphics[width=\linewidth]{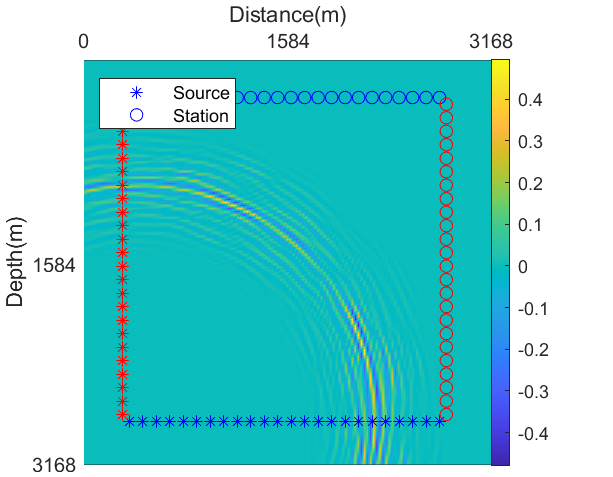}
			\end{minipage}%
		}%
		\subfigure[]{
			\label{fig:2Dfwi2_kernel}
			\begin{minipage}[t]{0.32\linewidth}
				\centering
				\includegraphics[width=\linewidth]{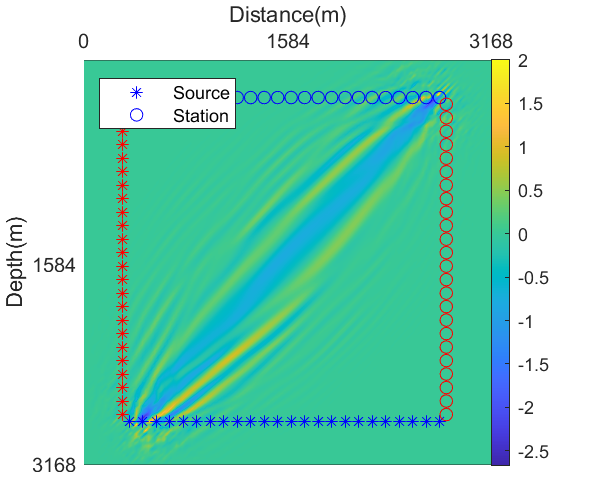}
			\end{minipage}%
		}%
		\caption{Example \ref{sec:fwi2d}. 2D FWI results with batch strategy 2 using the sampling rate $2.5\%$. After three iterations, we have: (a) velocity; (b) wavefield; (c) kernel. }
	\end{figure}
	\begin{figure}
		\label{fig:2Dfwimis}
		\centering
		\subfigure[]{
			\begin{minipage}[t]{0.6\linewidth}
				\centering
				\includegraphics[width=\linewidth]{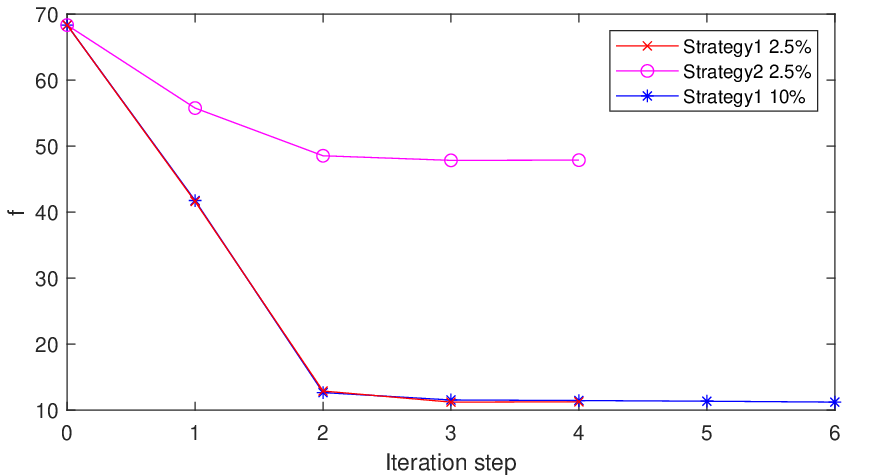}
			\end{minipage}%
		}%
		\caption{Example \ref{sec:fwi2d}. 2D FWI results: Decay of the misfit function.}
	\end{figure}
	
		We can see from FIG. \ref{fig:2Dfwi1}(b)(e) and \ref{fig:2Dfwi2}(b) that different strategies of batch generation in FGA catch similar shape of wavefront, this is because FGA as a ray-based asymptotic method gives correct ray-path information; the wavefields are smooth because they are reconstructed by complex-valued Gaussian functions. FIG. \ref{fig:2Dfwi1}(c)(f) and \ref{fig:2Dfwi2}(c) also show that the shape of kernels also looks similar as  Banana-Doughnuts. But one can see significant difference if looking at small scale structures: apparently there is roughness or ``noise-like" structure in FIG. \ref{fig:2Dfwi1}(f) generated by strategy 1, while FIG. \ref{fig:2Dfwi2}(c) by strategy 2 shows a smooth kernel. The reason for differences lies in that strategy 2 has no randomness in time evolution and so it is smooth when integrating (or summing, for numerical purpose) over time to get the kernel \eqref{eq:k_rb}, while strategy 1 use independently random batch for each time evolution step and so it shows more randomness.
		We should address here that it is the time-independence that helps strategy 1 attains a better convergence than strategy 2, which can be seen in the proof of Lemma \ref{lemma:var} that the computation of variance \eqref{eqn:var_1st} relies on the independence directly.

	
	\subsection{Travel-time inversion for a 2D gradually changing background model}\label{sec:tti2d}
	{
		The perturbation in the velocity field is small (3\%) in the previous subsection.
		As it has been observed in literature (\eg \cite{chai2018tomo}), the travel-time inversion has a much wider convergence zone than the full-waveform inversion, so when the perturbation is large, one can use travel-time inversion instead of full-waveform inversion to get a convergent result.
		To further test the performance of the proposed method, we look at a region with gradually changing background velocity and aim to image a target of low-velocity perturbation using travel-time inversion. As shown in FIG. \ref{fig:tti_var}(a), 48 stations are put near the top ground, 24 sources are put deep inside the earth near bottom of the target region of size $6336\xunit~\times~3168\xunit$, and the background velocity field has a lightly graduate change from 2500\cunit\; at the top ground to 3000\cunit\; at the deep bottom, and the velocity field is given by 
		\begin{equation}\label{eqn:2dgraduallychangemodel}
    		c(x,z) = \left(C_1(1-\frac{z}{2L}) + C_2\frac{z}{2L}\right) \left(1-\alpha\exp\left(-\f\beta{L^2}\left((x-x_c)^2+(z-z_c)^2\right)\right)\right),
	    \end{equation}
	    where $C_1=2500\cunit$, $C_2=3000\cunit$, $z_c=L=1584\xunit$,$x_c=3168\xunit$, $\alpha=0.1$, $\beta=24.2$.
	    travel-time inversion iteration starts with the background velocity $c_0$ which is given by \eqref{eqn:2dgraduallychangemodel} with the same parameter values specified above except that $\alpha=0$. In this numerical example, the beam number $N=32768$,and $\epsilon = L/256 $.}  We use FGA to simulate the forward and adjoint wave equations with sampling rate $p/N=$20\% to generate random batch for wavefield and kernel reconstruction. In FIG. \ref{fig:tti_var}(b), we plot the resulted velocity model $c_{10}$ after 10 iteration steps, and one can see a good match with the target velocity model, this also can be seen in FIG. \ref{fig:tti_var}(c) where $c_{10}-c_0$ is plotted; FIG. \ref{fig:tti_var}(d) shows the decay of misfit functional and one can see the iteration has convergent numerically.
	    
		\begin{figure}[ht]
			\label{fig:tti_var}
			\centering
			\subfigure[]{
				\begin{minipage}[t]{0.5\linewidth}
					\centering
					\includegraphics[width=\linewidth]{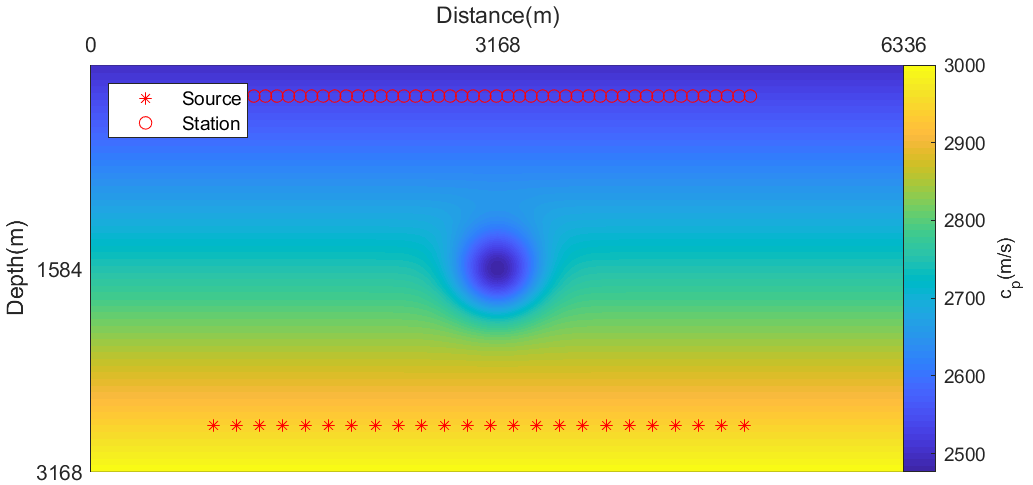}
				\end{minipage}%
			}%
			\subfigure[]{
				\begin{minipage}[t]{0.5\linewidth}
					\centering
					\includegraphics[width=\linewidth]{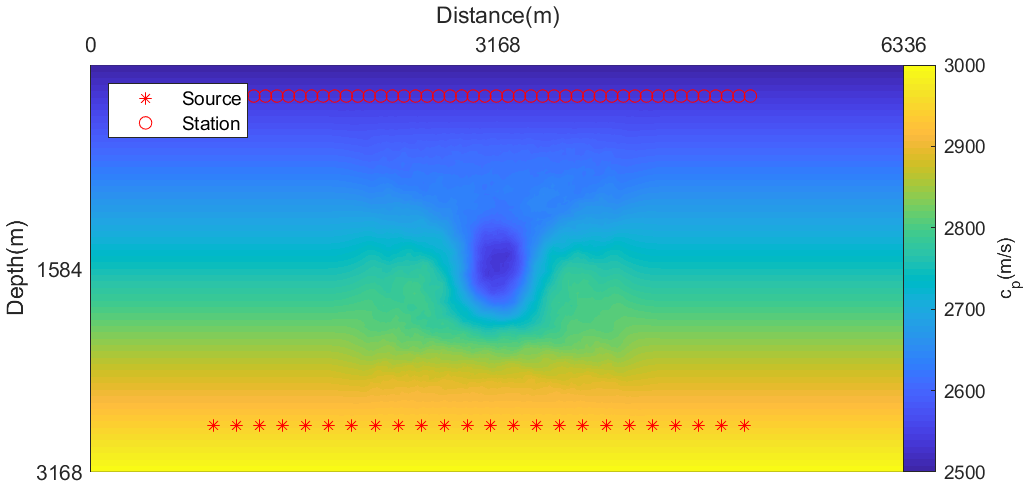}
				\end{minipage}%
			}%
			\quad
			\subfigure[]{
				\begin{minipage}[t]{0.5\linewidth}
					\centering
					\includegraphics[width=\linewidth]{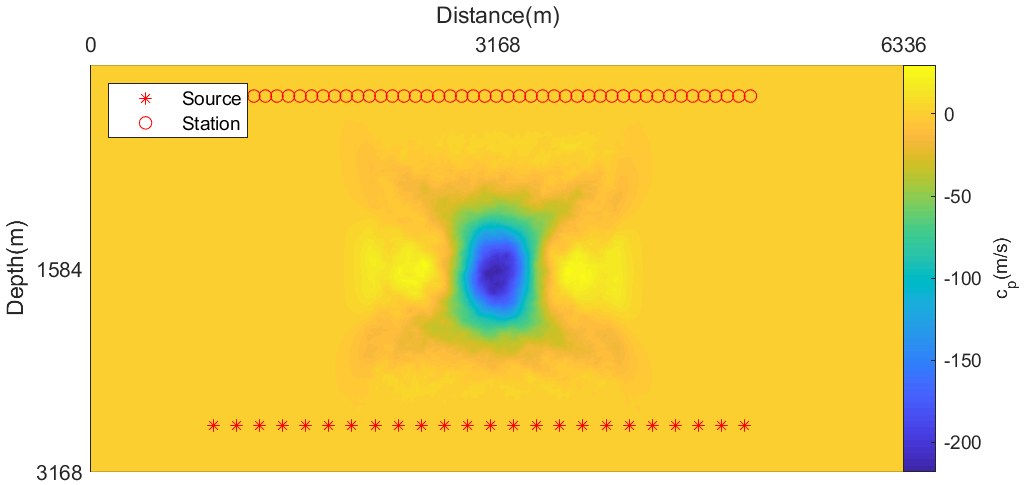}
				\end{minipage}%
			}%
			\subfigure[]{
				\begin{minipage}[t]{0.5\linewidth}
					\raggedright
					\includegraphics[width=0.98\linewidth]{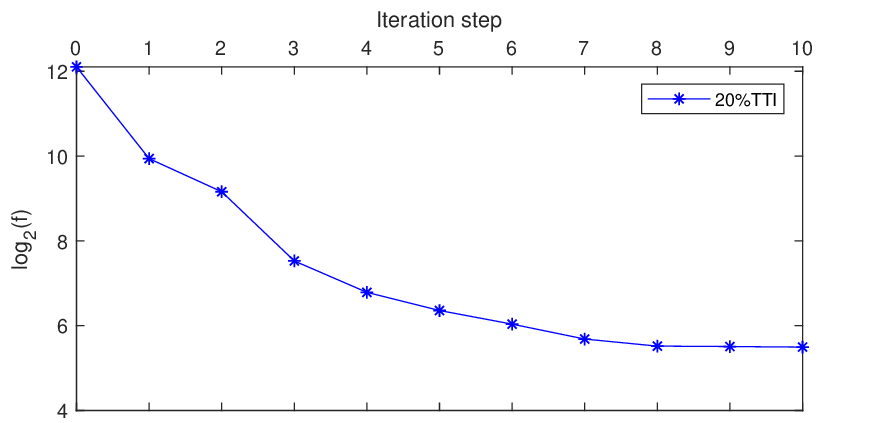}
					\label{fig:tti_var_misfit}
				\end{minipage}%
			}%
			\caption{Example \ref{sec:tti2d}. 2D TTI Model results with graduate change background: (a) target velocity field; (b) velocity field after 10 iteration steps; (c) velocity field after 10 iteration steps subtract the background velocity field $c_{10}-c_0$; (d) decay of the misfit function.}
		\end{figure}
		

	\subsection{Travel-time inversion for a three-layered model}\label{sec:tti3layer}
	
	In this example, we apply the travel-time inversion with random batch gradient reconstruction in a cross-well setup which is often used for high-resolution reservoir characterization in exploration geophysics. Two wells with 24 sources and 48 stations respectively are set on the left and right side of a region of size $3168\xunit~\times~6436\xunit$. The target velocity model is chosen as three-layered in the form of
	\begin{align}\label{eqn:3layer}
		c(x,y,z)=c(x,z)=
		\begin{cases}
			C_1, & \text{ if } z_0 < z < z_1, \\
			C_2\left(1-\alpha\exp\left(-\f\beta{L^2}\left((x-x_c)^2+(z-z_c)^2\right)\right)\right), & \text{ if } z_1 < z < z_2,\\
			C_3, & \text{ if } z > z_2,
		\end{cases}
	\end{align}
	where the background velocities in three layers are $C_1=1800\cunit$, $C_2=2000\cunit$, $C_3=2200\cunit$, and the layer interfaces locate at $z_0=0\xunit$, $z_1=2112\xunit$, $z_2=4224\xunit$. A low-velocity region characterized by a Gaussian perturbation is located at center  of the second layer with $x_c=1584\xunit$, $z_c=3168\xunit$, and we choose $\alpha=10\%$, $\beta=24.2$ and $L=3168\xunit$.
	See Fig.~\ref{fig:tti_3layer}(a) for an illustration of the velocity model and the source-receiver setup. In this numerical example, the beam number $N=32766$,and $\epsilon = L/256 $. 
	Travel-time inversion iteration starts with an piece-wise constant background velocity which is given by \eqref{eqn:3layer} with the same parameter values specified above except that $\alpha=0$.
	We use FGA to simulate the forward and adjoint wave equations with sampling rate $p/N=$20\% to generate random batch for wavefield and kernel reconstruction. In FIG. \ref{fig:tti_3layer}(b), we plot the resulted velocity model $c_9$ after 9 iteration steps, and one can see a good match with the target velocity model. FIG. \ref{fig:tti_3layer}(c) shows the difference in resulted and target velocity models $c_9-c$, and one can see the residual is relatively small comparing to the background and the Gaussian perturbation. 
	Decay of the misfit during iterations is shown in FIG. \ref{fig:tti_3layer_misfit} and one can see that the iteration has convergent numerically.
	\begin{figure}[ht]
		\label{fig:tti_3layer}
		\centering
		\subfigure[]{
			\begin{minipage}[t]{0.25\linewidth}
				\centering
				\includegraphics[width=\linewidth]{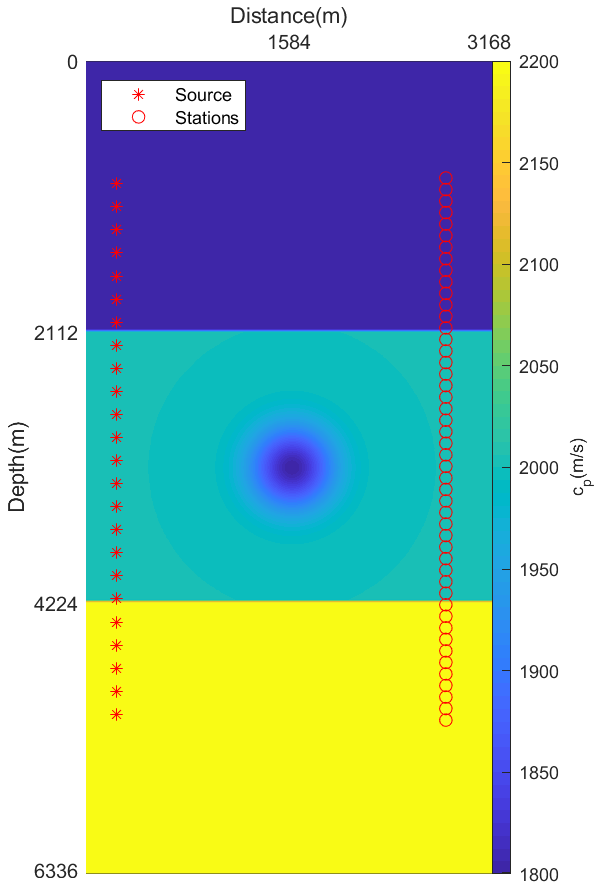}
			\end{minipage}%
		}%
		\subfigure[]{
			\begin{minipage}[t]{0.25\linewidth}
				\centering
				\includegraphics[width=\linewidth]{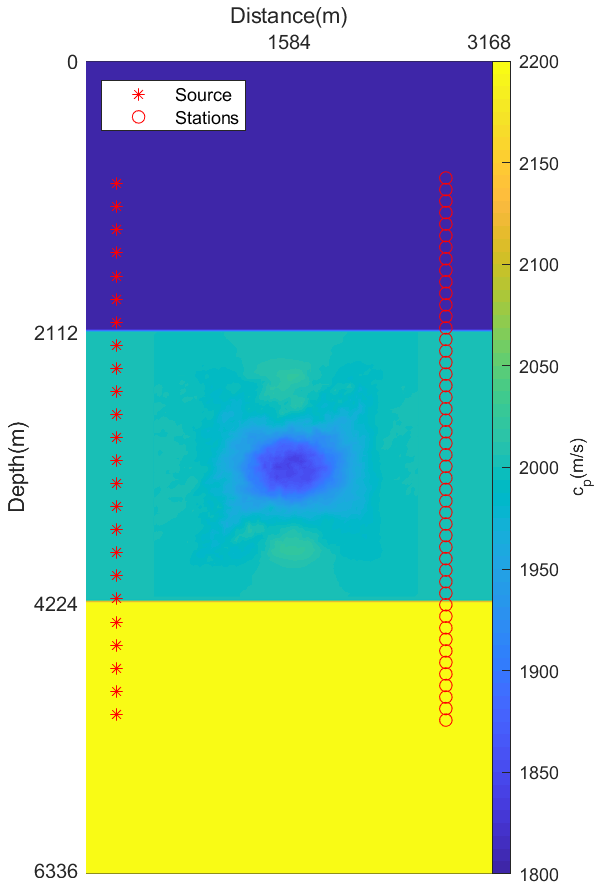}
			\end{minipage}%
		}%
		\subfigure[]{
			\begin{minipage}[t]{0.25\linewidth}
				\centering
				\includegraphics[width=\linewidth]{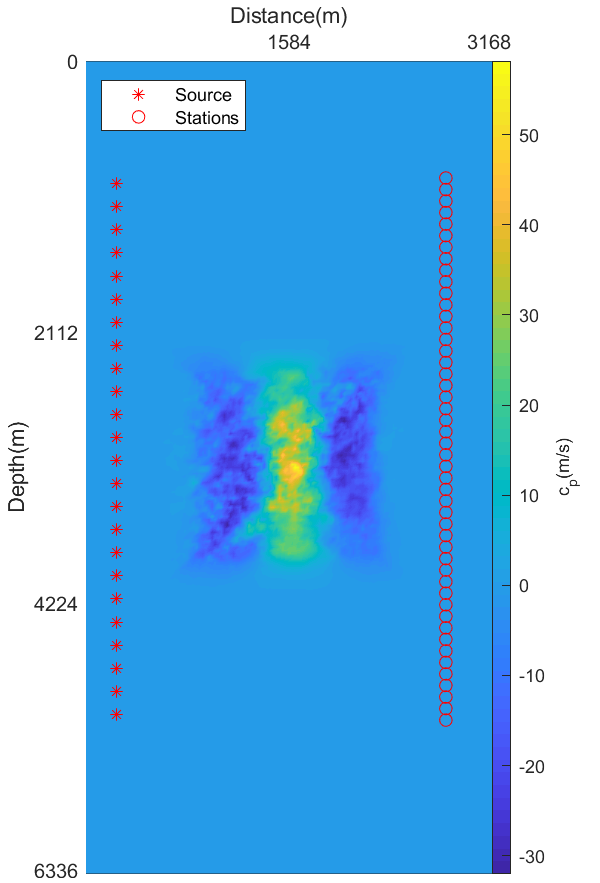}
			\end{minipage}%
		}%
		\caption{Example \ref{sec:tti3layer}. 2D TTI results for a three-layered model: (a)target velocity. (b) velocity after 9 iteration steps. (c) the difference between velocity after 9 iteration steps and target velocity. }
	\end{figure}
	
	\begin{figure}[!ht]
		\label{fig:tti_3layer_misfit}
		\centering
		\includegraphics[width=0.6\linewidth]{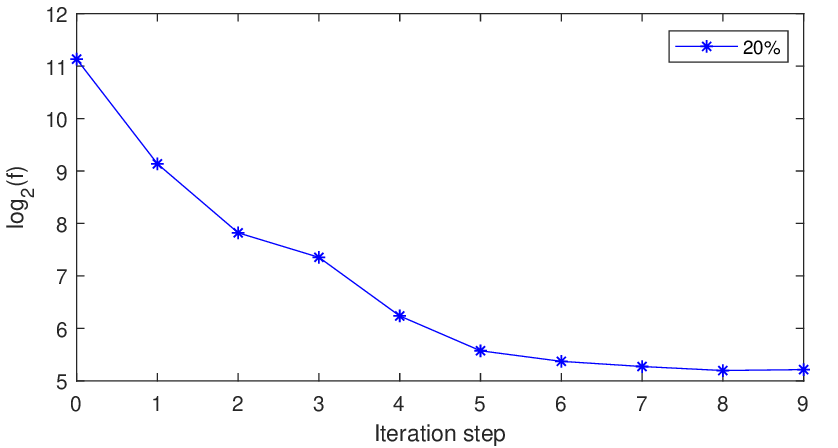}
		\caption{Example \ref{sec:tti3layer}. 2D TTI results for a three-layered model: Decay of the misfit function.}
	\end{figure}

	\subsection{Travel-time inversion for a 3D model}\label{sec:tti3d}
	In this subsection, we present a test using travel-time inversion to image a 3D cube region. The target velocity field is assumed homogeneous in $y$-direction and is set to be a 10\% Gaussian perturbation of a homogeneous background with velocity $2500\cunit$, \ie, 
	\begin{equation}\label{eqn:3dmodel}
		c(x,y,z)=c(x,z) = C_0\left(1-\alpha\exp\left(-\f\beta{L^2}\left((x-x_c)^2+(z-z_c)^2\right)\right)\right),
	\end{equation}
	where $C_0=2500\cunit$, $x_c=z_c=L=1584\xunit$, $\alpha=0.1$, $\beta=24.2$.  Travel-time inversion iteration starts with an velocity field with $2500\cunit$ homogeneously. In this numerical example, the beam number $N=65534$,and $\epsilon = L/64 $. 
	We use FGA to simulate the forward and adjoint wave equations in 3D. To reconstruct the wavefields and kernels, we use strategy 1 with sampling rate $p/N=$5\%, 10\% and 20\% to generate random batch for wavefield and kernel reconstruction.
	

	\begin{figure}[ht]
		\label{fig:3Dtti}
		\centering
		\subfigure[]{
			\begin{minipage}[t]{0.45\linewidth}
				\centering
				\includegraphics[width=\linewidth]{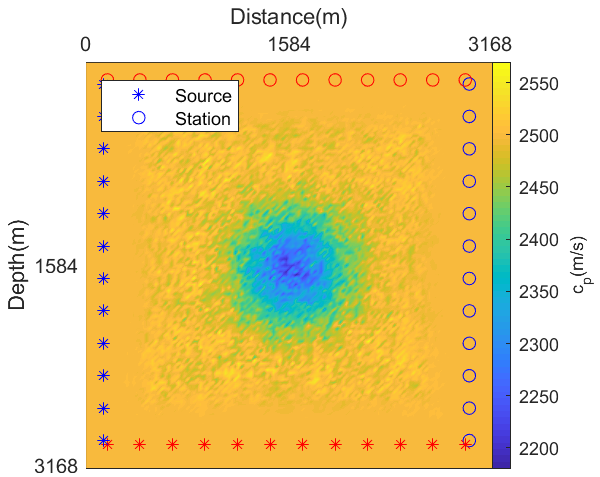}
			\end{minipage}%
		}%
		\subfigure[]{
			\begin{minipage}[t]{0.45\linewidth}
				\centering
				\includegraphics[width=\linewidth]{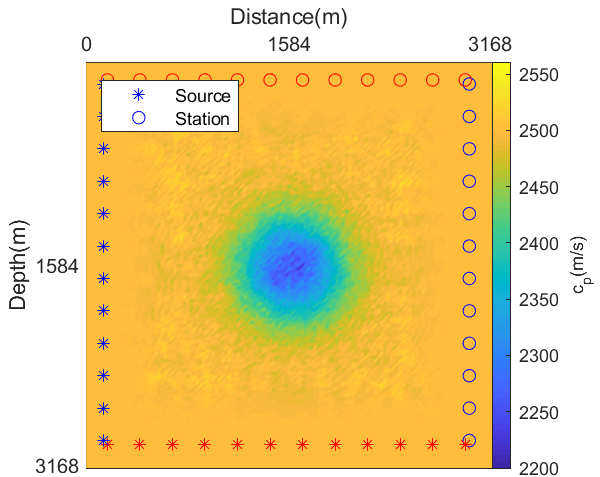}
			\end{minipage}%
		}%
		\quad
		\subfigure[]{
			\begin{minipage}[t]{0.45\linewidth}
				\centering
				\includegraphics[width=\linewidth]{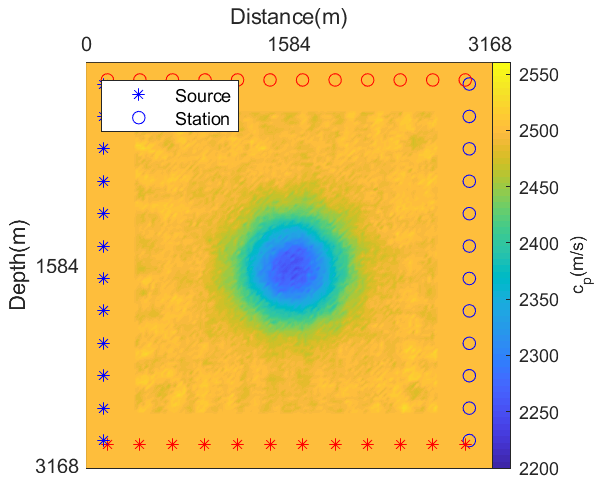}
			\end{minipage}%
		}%
		\subfigure[]{
			\begin{minipage}[t]{0.45\linewidth}
				\centering
				\includegraphics[width=\linewidth]{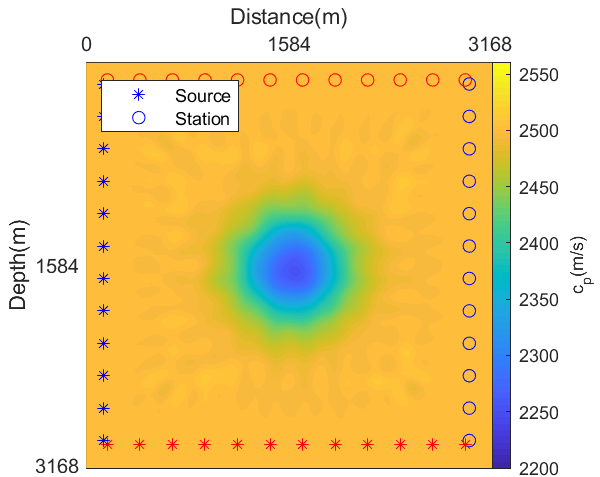}
			\end{minipage}%
		}%
		\caption{Example \ref{sec:tti3d}. 3D TTI results: (a) velocity model after 9 iteration step with sampling rate 5\%; (b) velocity model after 7 iteration step with sampling rate 10\%; (c) velocity model after 7 iteration step with sampling rate 20\%; (d) velocity model after 9 iteration step with sampling rate 100\%.}
	\end{figure}
	\begin{figure}[ht]
		\label{fig:3Dtti_misfit}
		\centering
		\includegraphics[width=0.6\linewidth]{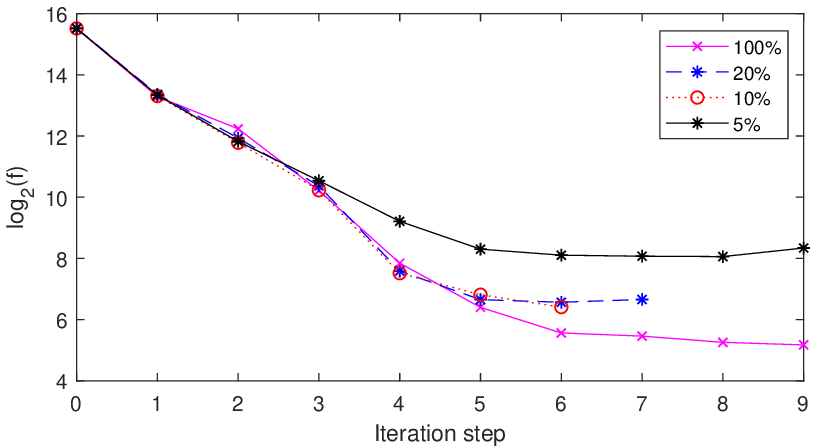}
		\caption{Example \ref{sec:tti3d}. 3D TTI results: Decay of misfit function.}
	\end{figure}
	
	It can be seen from FIG. \ref{fig:3Dtti} that even when 5\% of the beams are used, the tomography can capture at least the location and shape of the low velocity region and give a reasonable model. The larger of the sampling rate, the better resolution of the resulted image. This phenomena can be further seen in FIG. \ref{fig:3Dtti_misfit} as we look at the values of misfit functional. One can observe that smaller misfit value for larger sampling rate, which consists with Theorem \ref{thm:1} since lager $p/N$ implies smaller variance in \eqref{eqn:var} for fixed $N$.  On the other hand, for the first several iteration steps, the value of misfit functions are almost the same for different sampling rate which numerically indicate the convergence rate of the iteration is not sensitive to the batch size.

{
\subsection{Travel-time inversion for a 3D refraction model}\label{sec:exp_gd}
    In this subsection, we give an example using travel-time inversion to image a 3D cuboid region with a completely reflective bottom interface. The target velocity field is assumed homogeneous in $y$-direction and is set to be a 10\% Gaussian perturbation of a homogeneous background with velocity $2500\cunit$, \ie, 
	\begin{equation}\label{eqn:3dmodel_reflect}
		c(x,y,z)=c(x,z) = C_0\left(1-\alpha\exp\left(-\f\beta{L^2}\left((x-x_c)^2+(z-z_c)^2\right)\right)\right),
	\end{equation}
	where $C_0=2500\cunit$, $x_c=L=1584\xunit$, $z_c=891\xunit$, $\alpha=0.1$, $\beta=6.05$.  Travel-time inversion iteration starts with an velocity field with $2500\cunit$ homogeneously.  In this numerical example, the beam number $N=65534$,and $\epsilon = L/64 $. 
	We use FGA to simulate the forward and adjoint wave equations in 3D. To reconstruct the wavefields and kernels, we use strategy 1 with sampling rate $p/N=$5\%, 20\%, and 100\% to generate random batch for wavefield and kernel reconstruction.

	\begin{figure}[ht]
		\label{fig:tti_reflect}
		\centering
		\subfigure[]{
			\begin{minipage}[t]{0.5\linewidth}
				\centering
				\includegraphics[width=\linewidth]{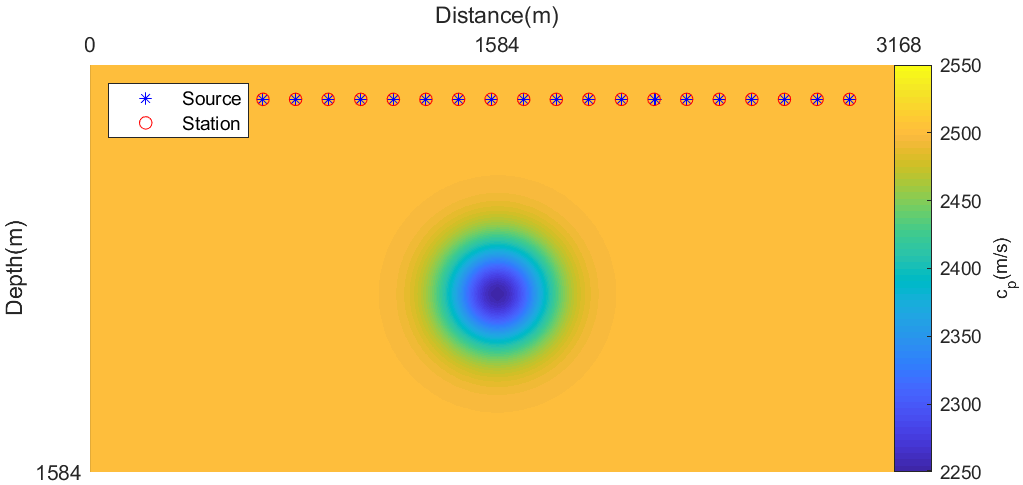}
			\end{minipage}%
		}%
		\subfigure[]{
			\begin{minipage}[t]{0.5\linewidth}
				\centering
				\includegraphics[width=\linewidth]{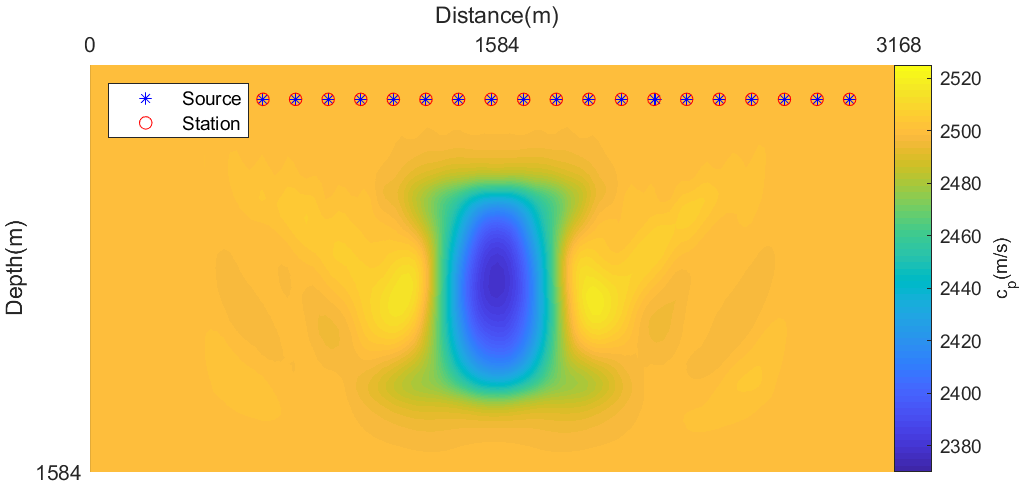}
			\end{minipage}%
		}%
		\quad
		\subfigure[]{
			\begin{minipage}[t]{0.5\linewidth}
				\centering
				\includegraphics[width=\linewidth]{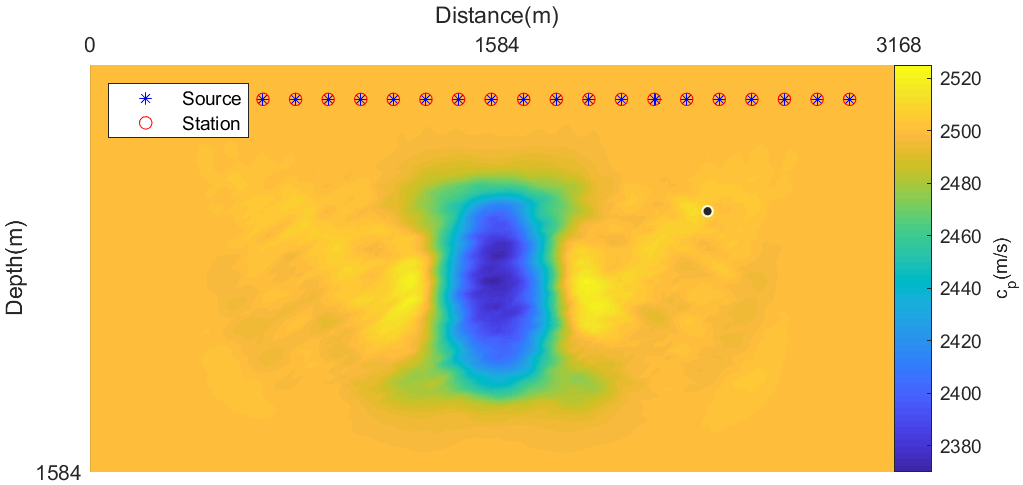}
			\end{minipage}%
		}%
		\subfigure[]{
			\begin{minipage}[t]{0.5\linewidth}
				\centering
				\includegraphics[width=\linewidth]{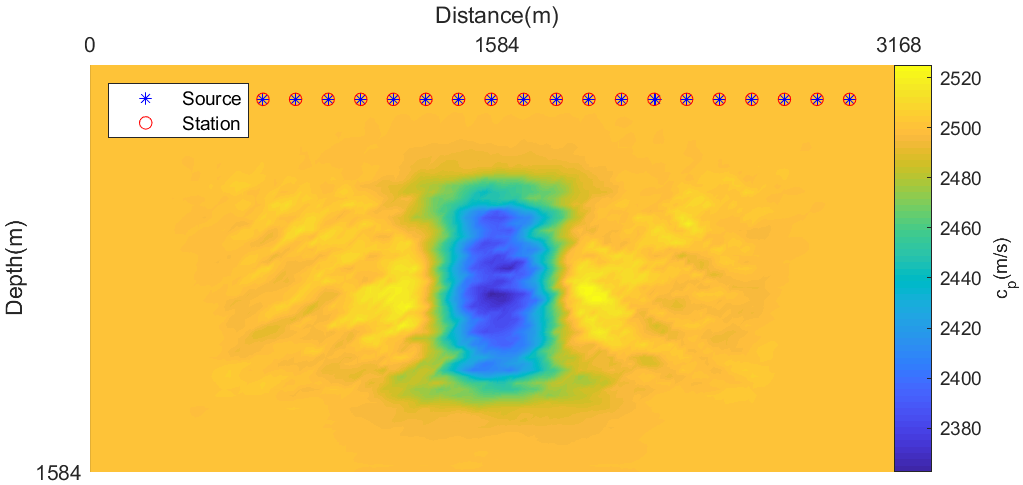}
			\end{minipage}%
		}%
		\caption{Example \ref{sec:exp_gd}. 3D TTI Model results with reflection bottom: (a) target velocity field; (b) velocity field after 14 iteration steps with sampling rate 100\%; (c) velocity field after 11 iteration steps with sampling rate 20\%; (d) velocity field after 15 iteration steps with sampling rate 5\%; }
	\end{figure}
	
	\begin{figure}[ht]
		\label{fig:3Dtti_reflection_misfit}
		\centering
		\includegraphics[width=0.6\linewidth]{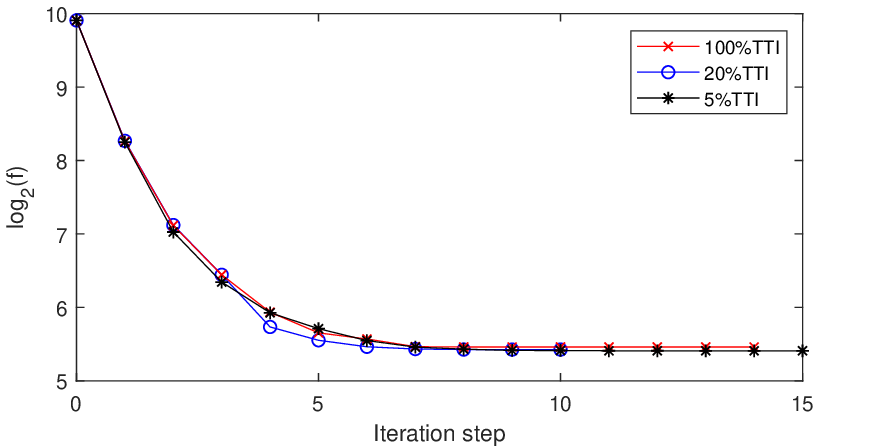}
		\caption{Example \ref{sec:exp_gd}. 3D TTI Model results with reflection bottom: Decay of misfit function.}
	\end{figure}

	It can be seen from FIG. \ref{fig:3Dtti_reflection_misfit} that, in this test, the tomography can capture at least the location and shape of the low velocity region and give a reasonable model. The result with different sampling rates are nearly the same. This phenomena can be further observed in FIG. \ref{fig:3Dtti_reflection_misfit} as we look at the values of misfit function. 
	
	We can also compare the computation time to show the effectiveness of the stochastic method in saving computations.
	As one can see from TABLE \ref{tab}, for different sampling rate from 100\% to 5\%, the computation time spent for 1 iteration step varies from 30.08 hours to 4.18 hours, which indicates that it can save about 86.1\% CPU time by using the random batch method.
	
	\begin{table}
	    \centering
	    \label{tab}
	    Efficiency of random batch method\\
	    \begin{tabular}{|c|c|c|}
	        \hline
	       Sampling Rate  & Computation Time & Saving \\
	       \hline
	       100\%  & 30.08h & 0\% \\
	       \hline
	       20\% & 7.65h & 74.6\% \\
	       \hline
	       10\% & 5.08h & 83.1\%  \\
	       \hline
	       5\% & 4.18h & 86.1\%  \\
	       \hline
	    \end{tabular}
	    \caption{Table of efficiency for different sampling rates.}
	\end{table}
	
}

	\section{Conclusion and discussion}\label{sec:conclusion}
    In this paper, we propose a type of stochastic gradient descent method for seismic tomography. Specifically, we use the frozen Gaussian approximation (FGA) to compute seismic wave propagation, and then construct stochastic gradients by random batch methods. One can easily generalize the idea by replacing FGA with any other efficient wave propagation solvers, {\it e.g.}, Gaussian beam methods. The convergence of the proposed method is proved in the mean-square sense, and we present four examples of both wave-equation-based travel-time inversion and full-waveform inversion to show the numerical performance. This method provides a possibility of efficiently solving high-dimensional optimization problems in seismic tomography. We plan to apply it to image subsurface structures of Earth using realistic seismic signals. 
	
	\section*{Acknowledgements}
	L.C. was partially supported by 
	the National Key R\&D Program of China No. 2021YFA1003001, 2020YFA0712503, 
	the NSFC Projects No. 12271537 and 11901601, 
	and the Fundamental Research Funds for the Central Universities, Sun Yat-sen University No. 2021qntd21. Z.H. was partially supported by the NSFC Projects No. 12025104, 11871298, 81930119. X.Y. was partially supported by the NSF grants DMS-1818592 and DMS-2109116.

	\bibliographystyle{siamplain}
	\bibliography{fga_ref,paper}

\end{document}